\documentclass{amsart}

  \usepackage{latexsym}

\usepackage{amsfonts,amssymb,euscript,epsfig,color}

\usepackage[all]{xy}

\newtheorem{theo}{Theorem}[section]
\newtheorem{defi}[theo]{Definition}

\newtheorem{lem}[theo]{Lemma}
\newtheorem{prop}[theo]{Proposition}
\newtheorem{rem}[theo]{Remark}

\newtheorem{exam}[theo]{Example}

\newcommand{\hgot}{\ensuremath{\mathfrak{h}}}
\newcommand{\kgot}{\ensuremath{\mathfrak{k}}}

\newcommand{\tgot}{\ensuremath{\mathfrak{t}}}

\newcommand{\zgot}{\ensuremath{\mathfrak{z}}}

\newcommand{\Bcal}{\ensuremath{\mathcal{B}}}
\newcommand{\Ccal}{\ensuremath{\mathcal{C}}}
\newcommand{\Dcal}{\ensuremath{\mathcal{D}}}
\newcommand{\Ecal}{\ensuremath{\mathcal{E}}}

\newcommand{\Lcal}{\ensuremath{\mathcal{L}}}

\newcommand{\Ncal}{\ensuremath{\mathcal{N}}}
\newcommand{\Ocal}{\ensuremath{\mathcal{O}}}

\newcommand{\Qcal}{\ensuremath{\mathcal{Q}}}
\newcommand{\Scal}{\ensuremath{\mathcal{S}}}
\newcommand{\Xcal}{\ensuremath{\mathcal{X}}}

\newcommand{\Zcal}{\ensuremath{\mathcal{Z}}}
\newcommand{\Ucal}{\ensuremath{\mathcal{U}}}

\newcommand{\Z}{\ensuremath{\mathbb{Z}}}
\newcommand{\C}{\ensuremath{\mathbb{C}}}

\newcommand{\R}{\ensuremath{\mathbb{R}}}

\newcommand{\tore}{\ensuremath{\mathbb{T}}}
\newcommand{\Lbb}{\ensuremath{\mathbb{L}}}

\newcommand{\mm}{\ensuremath{\hbox{\rm m}}}
\newcommand{\tr}{\ensuremath{\hbox{\bf Tr}}}

\newcommand{\esp}{\ensuremath{\varepsilon}}

\newcommand{\fgene}{\ensuremath{\mathcal{C}^{-\infty}}}
\newcommand{\croc}{\ensuremath{\hookrightarrow}}

\newcommand{\indice}{\ensuremath{\hbox{\rm Index}}}

\newcommand{\p}{\ensuremath{\Delta}}
\newcommand{\pf}{\ensuremath{\overrightarrow{\Delta}}}
\newcommand{\Cr}{\ensuremath{\hbox{\rm Cr}}}

\newcommand{\Char}{\ensuremath{\hbox{\rm Char}}}

\newcommand{\Thom}{\ensuremath{\hbox{\rm Thom}}}

\newcommand{\Rfor}{\ensuremath{R^{-\infty}}}
\newcommand{\Rforc}{\ensuremath{R^{-\infty}_{tc}}}

\newcommand{\spinq}{\ensuremath{\Qcal_{\hbox{\tiny spin}}}}

\newcommand{\spin}{\ensuremath{\rm Spin}}
\newcommand{\spinc}{\ensuremath{{\rm Spin}^{c}}}

\def \K {{\rm \bf K}}
\def \T {{\rm T}}

\def \what {\widehat}

\def \indB {{\rm Ind}^{^K}_{_{K_{\beta}}}}
\def \indHa {{\rm Ind}^{^H}_{_{H_a}}}

\def \clif {{\bf c}}

\begin{document}

\title{Spin-quantization commutes with reduction}

\author{Paul-Emile  PARADAN}

\address{Institut de Math\'ematiques et de Mod\'elisation de Montpellier (I3M), 
Universit\'e Montpellier 2} 

\email{Paul-Emile.Paradan@math.univ-montp2.fr}

\date{November 2009}

\begin{abstract}
In this paper, we prove that the ``quantization commutes with reduction'' phenomenon 
of Guillemin-Sternberg \cite{Guillemin-Sternberg82} applies in the context 
of the metaplectic correction. 
\end{abstract}

\maketitle

{\def\thefootnote{\relax}
\addtocounter{footnote}{-1}
}

{\small
\tableofcontents}


\section{Introduction}

Let $K$ be a compact connected Lie group with Lie algebra $\kgot$. 
An Hamiltonian K-manifold $(M,\omega,\Phi)$ is Spin-prequantized if $M$ carries an equivariant $\spinc$ 
structure $P$ with determinant line bundle being a Kostant-Souriau line bundle over $(M,2\omega,2\Phi)$. 
Let $\Dcal_P$ be the $\spinc$ Dirac operator attached to $P$, where $M$ is oriented by its symplectic form. 
The $\spin$ quantization of $(M,\omega,\Phi)$ corresponds to the equivariant index of the elliptic 
operator $\Dcal_P$, and is denoted 
$$
\spinq^K(M)\ \in\ R(K).
$$

Let $\widehat{A}(M)(X)$ be the equivariant A-genus class: it is an equivariant analytic function from a neigborhood 
of $0\in\kgot$ with value in the algebra of differential forms on $M$. The Atiyah-Segal-Singer index 
theorem \cite{B-G-V} tell us that 
\begin{equation}\label{eq:ASS-index}
\spinq^K(M)(e^X):=\int_M e^{i(\omega+\langle\Phi,X\rangle)}\widehat{A}(M)(X)
\end{equation}
for $X\in\kgot$ small enough. It shows in particular that $\spinq^K(M)\in R(K)$ does not depend of 
the choice of the Spin-prequantum data.

This notion of Spin-quantization is closely related to the notion of {\em metaplectic correction}. Suppose that  
$(M,\omega,\Phi)$ carries a Kostant-Souriau line bundle $L_\omega$, and that the bundle of half-forms 
$\kappa_J^{1/2}$ associated to an invariant almost complex structure $J$ is well defined. 
In this case, $(M,\omega,\Phi)$ is Spin-prequantized 
by the $\spinc$-structure defined by $J$ and twisted by the line bundle $L_\omega\otimes\kappa_J^{1/2}$. The crucial
point here is that the corresponding Spin-quantization of $(M,\omega,\Phi)$ does not depend of the choice 
of the almost complex structure. Note that the existence of the bundle of half-form $\kappa_J^{1/2}$ is equivalent 
to the existence of a Spin structure on $M$ \cite{Lawson-Michel}.

The purpose of this article is to compute geometrically the multiplicities of $\spinq^K(M)\ \in\ R(K)$ in a way 
similar to the famous ``quantization commutes with reduction'' phenomenon of Guillemin-Sternberg 
\cite{Guillemin-Sternberg82,Meinrenken98,Meinrenken-Sjamaar,Tian-Zhang98,pep-RR,Vergne96,Jeffrey-Kirwan97,Vergne-Bourbaki,Sjamaar96}. 
This question was partially resolved 
in a previous paper \cite{pep-ENS} under the  condition that the infinitesimal stabilizers of the 
$K$-action are abelian. C. Teleman also obtained some results \cite{Teleman-Annals00}[Proposition 3.10] 
in the algebraic setting.

The striking difference with the standard Guillemin-Sternberg phenomenon is the {\em rho shift} that we 
explain  now. Let $T$ be a maximal torus of $K$ with Lie algebra $\tgot\subset\kgot$. 
Let $\tgot^*_+\subset\tgot^*$ be the closed Weyl chamber. We will 
look at $\tgot^*_+$ as a disjoint union of its open faces, the maximal one being its interior $(\tgot^*_+)^o$. 
Let $\rho\in (\tgot^*_+)^o$ be the half sum of the positive roots. At each open face $\tau$ of $\tgot^*_+$, 
we associate the term $\rho_\tau$ which is the half sum of the positive roots which are orthogonal to $\tau$. 
We note that $\rho-\rho_\tau\in \tau$ is to the orthogonal projection of $\rho$ on $\tau$. 

For any $\xi \in\tgot^*_+$ and any face $\tau$ containing $\xi$ in its closure, we consider 
the {\em shifted symplectic reduction} 
$$
M_\xi^\tau:=\Phi^{-1}(\xi+\rho-\rho_\tau)/K_\tau
$$
where $K_\tau$ is the common stabiliser of points in $\tau$. Note that $\xi+\rho-\rho_\tau\in \tau$ when $\xi\in\overline{\tau}$.

We are particularly interested to the  smallest face $\sigma$ of the Weyl chamber so that the 
Kirwan polytope $\p(M):=\Phi(M)\cap\tgot^*_+$ is contained in the closure of $\sigma$. It is not 
hard to see that the Spin-prequantum data on $(M,\omega,\Phi)$ descents to 
the shifted symplectic reduction $M_\mu^\sigma$ when $\mu$ is a dominant weight 
belonging to $\overline{\sigma}$. Then $\spinq(M_\mu^\sigma)\in\Z$ is naturally defined when 
$\mu+\rho-\rho_\sigma$ is a regular value of the moment map.  In general, the 
number $\spinq(M_\mu^\sigma)$ is defined by shift-desingularization (see Section \ref{subsec:spin-prequantization-reduction}).  

By definition $\spinq(M_\mu^\sigma)$ vanishes when $\mu+\rho-\rho_\sigma\notin\p(M)$, but in fact 
we can strengthen this vanishing property: $\spinq(M_\mu^\sigma)=0$ if $\mu+\rho-\rho_\sigma$ does not 
belong to the {\em relative interior} of the Kirwan polytope $\p(M)$.

Recall that the irreducible representation $V_\mu^K$ of $K$ are parametrized by their highest weight 
$\mu\in\what{K}\subset\tgot^*_+$.

The main result of this paper is the following

\begin{theo}\label{theo:QRspin-intro}
Let $(M,\omega,\Phi)$ be a compact Spin-prequantized Hamiltonian $K$-manifold. 
 Let $\sigma$ be the smallest face of the Weyl chamber so that 
 $\p(M)\subset\overline{\sigma}$. We have
 $$
 \Qcal_{\hbox{{\rm \tiny spin}}}^K(M)= \sum_{\mu\in\what{K}\cap\overline{\sigma}}
\Qcal_{\hbox{{\rm \tiny spin}}}(M_\mu^\sigma)\, V_\mu^K.
 $$
\end{theo}

\bigskip

Let us give some ideas about the proof. The representation $V_\mu^K$ 
is equal to the Spin-quantization of the coadjoint orbit $\Ocal_\mu:= K\cdot(\mu+\rho)$. 
Then the shifting trick tells us that the multiplicity $\mm_\mu$ of $V_\mu^K$ in $\spinq^K(M)$ is 
equal to 
$$
\left[\spinq^K(M)\otimes (V_\mu^K)^*\right]^K= 
\left[\spinq^K(M\times\overline{\Ocal}_\mu)\right]^K,
$$
where $\overline{\Ocal}_\mu$ is the coadjoint orbit with the opposite symplectic structure. 
As we did in \cite{pep-RR,pep-ENS}, we study the expression $\left[\spinq^K(M\times\overline{\Ocal}_\mu)\right]^K$
by localizing the Riemann-Roch character on the critical points of the square of the moment map 
$$
\Phi_\mu:M\times\overline{\Ocal}_\mu\to\kgot^*.
$$ 

Here our treatment differs depending on whether the Kirwan polytope $\p(M)$ intersects 
the interior of the Weyl chambers or not (i.e. $\sigma=\tgot^*_+$ or not).

When $\sigma=\tgot^*_+$, we show that the multiplicity $\mm_\mu$ is calculated using the Riemann-Roch character 
localized near the zero level set of the moment map $\Phi_\mu$. This case is (more or less) treated  in \cite{pep-ENS}.

The heart of this paper is when we work out the case $\sigma\neq\tgot^*_+$. We have $\Phi_\mu^{-1}(0)=\emptyset$, but 
we show how to compute $\mm_\mu$ using the Riemann-Roch character localized near 
$$
K\cdot \left(N^{\rho_\sigma}\cap \Phi_\mu^{-1}(-\rho_\sigma)\right).
$$
Here $N^{\rho_\sigma}$ denotes the submanifold of $N=M\times\Ocal_\mu$ where the infinitesimal action 
of $\rho_\sigma$ vanishes.

\bigskip

{\bf Notations.} Throughout the paper, $K$ will denote a compact connected Lie group, 
and $\kgot$ its Lie algebra. We let $T$ be a maximal torus in $K$, and $\tgot$ be its 
Lie algebra. The integral lattice $\wedge\subset\tgot$ is defined as the kernel of 
$\exp:\tgot\to T$, and the real weight lattice $\wedge^*\subset\tgot^*$ is defined by 
: $\wedge^*:=\hom(\wedge,2\pi\Z)$. Every $\mu\in\wedge^*$ defines a $1$-dimensional 
$ T$-representation, denoted $\C_\mu$, where $t=\exp(X)$ acts by $t^\mu:= e^{i\langle\mu, X\rangle}$. 
We fix a positive Weyl chamber $\tgot^*_+\subset\tgot^*$. For any dominant weight 
$\mu\in\what{K}:= \wedge^*\cap\tgot_+^*$, we denote by $V^K_\mu$ the irreducible 
representation with highest weight $\mu$. We denote $R(K)$ the representation ring of 
$K$. We denote $\Rfor(K):=\hom_\Z(R(K),\Z)$ its dual. An element $E\in\Rfor(K)$ can be represented 
as an infinite sum $E=\sum_{\mu\in\what{K}}\mm_\mu V^K_\mu$, with $\mm_\mu\in\Z$. The multiplicity $\mm_0$ 
of the trivial representation is denoted $[E]^K$. If $H$ is a closed subgroup of $K$, we have the induction map 
${\rm Ind}^{K}_{_H}: \Rfor(H)\to \Rfor(K)$ which is the dual of the restriction morphism 
$R(K)\to R(H)$. We see that $[{\rm Ind}^{K}_{_H}(E)]^K= [E]^H$.

\section{Spin-quantization of compact Hamiltonian $K$-manifolds}\label{sec.quantization.compact}

Let $M$ be a {\em compact} Hamiltonian $K$-manifold with symplectic form $\omega$ 
and moment map $\Phi: M\to \kgot^{*}$ characterized by the relation 
\begin{equation}\label{eq:hamiltonian-action}
    \iota(X_M)\omega= -d\langle\Phi,X\rangle,\quad X\in\kgot, 
\end{equation}
where $X_M(m):=\frac{d}{dt}\vert_{t=0} e^{-tX}\cdot m$ is the vector field on $M$ generated by $X\in \kgot$.

In the Kostant-Souriau framework \cite{Kostant70,Souriau}, a Hermitian line bundle $L_\omega$ with an invariant Hermitian connection
$\nabla$ is a prequantum line bundle over $(M,\omega,\Phi)$ if 
\begin{equation}\label{eq:kostant-L}
    \Lcal(X)-\nabla_{X_M}=i\langle\Phi,X\rangle\quad \mathrm{and} \quad
    \nabla^2= -i\omega,
\end{equation}
for every $X\in\kgot$. Here $\Lcal(X)$ is the infinitesimal action of $X\in\kgot$ on the sections 
of $L_\omega\to M$. $(L_\omega,\nabla)$ is also called a Kostant-Souriau line bundle. Remark that 
conditions (\ref{eq:kostant-L}) imply, via the equivariant
Bianchi formula, the relation (\ref{eq:hamiltonian-action}).

\subsection{Spin-quantization: definitions}\label{sec.prequantization.compact}

Let $J$ be any invariant almost complex structure on $M$, not necessarily compatible with the symplectic form 
$\omega$. Let 
$$
RR^{^K}_J(M,-)
$$
 be the corresponding Riemann-Roch character \cite{pep-RR}. 
We consider the complex tangent bundle $(\T M,J)$  and its complex dual 
$\T^{*}_{\C}M:=\hom_{\C}(\T M,\C)$. We consider the complex line bundle 
$$
\kappa_J:=\det\T^{*}_{\C}M.
$$

If $(M,\omega,\Phi)$ is prequantized by $L_\omega$, a standard 
procedure (called the metaplectic correction in the geometric quantization literature) is to tensor 
$L_\omega$ by the bundle of half-forms $\kappa_J^{1/2}$ \cite{Woodhouse}. 
We may consider the equivariant index 
\begin{equation}\label{eq:Q-J}
\Qcal_J^K(M):=\epsilon_J RR^{^K}_J(M,L_\omega\otimes\kappa_J^{1/2})
\end{equation}
where $\epsilon_J=\pm 1$ is the quotient of the orientations defined by $\omega$ and by $J$. In Proposition 
\ref{prop:spin-quantization} we check that  $\Qcal_J^K(M)$ has a meaning when the tensor product  $\tilde{L}=L_\omega\otimes\kappa_J^{1/2}$ is well defined  (even if neither $L_\omega$ nor $\kappa^{1/2}_J$ exist).

The almost complex structure $J$ defines a $\spinc$ structure $P_J$ on $M$ with determinant 
line bundle $\det P_J=\kappa_J^{-1}$. If we twist the $\spinc$ structure $P_J$ by any complex line bundle 
$\Lbb$ we get a $\spinc$ structure  $P_{J,\Lbb}$ with determinant line bundle 
$$
\det P_{J,\Lbb}=\Lbb^{2}\otimes\kappa_J^{-1}.
$$
See \cite{Lawson-Michel,pep-ENS}.

We make the following basic observation.

\begin{prop}\label{prop:spin-prequantized}
Let $(M,\omega,\Phi)$ be a Hamiltonian $K$-manifold. The following assertions 
are equivalent:
\begin{itemize}
\item[ a)] For any invariant complex structure $J$ there exists a $K$-equivariant line bundle $\tilde{L}$ 
    such that $\tilde{L}^{2}\otimes\kappa_J^{-1}$ is a prequantum
    line bundle over $(M,2\omega,2\Phi)$. 

\item[ b)] There exist an invariant complex structure $J$ and a $K$-equivariant line bundle $\tilde{L}$ 
    such that $\kappa_J^{-1}\otimes\tilde{L}^{2}$ is a prequantum
    line bundle over $(M,2\omega,2\Phi)$. 

\item[ c)] There exists an equivariant $\spinc$ structure $P$ such that its determinant line 
    bundle $\det P$ is a prequantum line bundle over $(M,2\omega,2\Phi)$.
\end{itemize}
\end{prop}

\medskip

When the previous assertions holds, we says that $(M,\omega,\Phi)$ is Spin-prequantized, either 
by the $\spinc$-structure $P$, or by the data $(J,\tilde{L})$. 
\medskip

\begin{prop}\label{prop:spin-quantization}
Let $(M,\omega,\Phi)$ be a $\spin$-prequantized Hamiltonian $K$-manifold. 
The equivariant index $\Qcal_J^K(M):=\epsilon_J RR^{^K}_J(M,\tilde{L})$ does not depend 
of the choice of the $\spin$-prequantum data $(J,\tilde{L})$. In fact $\Qcal_J^K(M)$ 
coincides with the equivariant index of the $\spinc$ Dirac operator $\Dcal_P$ attached to 
the $\spinc$-structure $P$.
\end{prop}

\medskip

\begin{defi}\label{def:Q-spin}
Let $(M,\omega,\Phi)$ be a $\spin$-prequantized Hamiltonian $K$-manifold. 
The $\spin$-quantization of $(M,\omega,\Phi)$ is defined as the equivariant index $\Qcal_J^K(M)$, and 
is denoted 
$$
\Qcal_{\hbox{{\rm \tiny spin}}}^K(M)\in R(K).
$$
\end{defi}

\medskip

{\em Proof of Propositions \ref{prop:spin-prequantized} and \ref{prop:spin-quantization}. } 
We have obviously $a)\Longrightarrow b)$, and we get $b)\Longrightarrow c)$ by taking 
the $\spinc$ sructure $P_{J,\tilde{L}}$. Let us prove $c)\Longrightarrow a)$. 

Let $P$ be a $\spinc$-structure on $M$ such that its determinant line bundle $\det P$ 
is a prequantum line bundle over $(M,2\omega)$. Let $\Scal_P$ be the corresponding bundle of spinors. 
 Let $P_J$ and $\Scal_J$ be respectively the 
associated $\spinc$-structure and the bundle of spinors on $M$ associated to an invariant almost complex structure 
$J$ on $M$. Since $\Scal_P, \Scal_J$ are irreducible clifford modules, we have 
\begin{equation}\label{eq:Scal}
\Scal_P\simeq S_J\otimes \tilde{L}
\end{equation}
where $\tilde{L}$ is the line bundle defined by $\tilde{L}:=\hom_{\rm cl}(\Scal_J, S_P)$. From 
(\ref{eq:Scal}) we get that the line bundle
\begin{eqnarray*}
\det P &= & \tilde{L}^2\otimes\det P_J\\
&=& \tilde{L}^2\otimes \kappa_J^{-1}.
\end{eqnarray*}
is a prequantum line bundle over $(M,2\omega,2\Phi)$.

Let $P$ be the $\spinc$ structure attached to a data $(J,\tilde{L})$. The symplectic orientation 
on $M$ defines a decomposition on the bundle of spinors, $\Scal_P=\Scal_P^+\oplus\Scal_P^-$, and 
the corresponding $\spinc$ Dirac operator $\Dcal_P$ maps $\Gamma(\Scal_P^+)$ to $\Gamma(\Scal_P^-)$.

On the other hand the almost complex structure on $M$ gives the decomposition $\wedge\T^*M\otimes\C=
\oplus_{i,j}\wedge^{i,j}\T^*M$ of the bundle of differential form. The corresponding bundle of 
spinors is $\Scal_J:=\wedge^{0, \bullet}\,\T^*M$ and the complex orientation induces the splitting 
$\Scal_J=\Scal_J^+\oplus\Scal_J^-$ with $\Scal_J^+:=\wedge^{0,even}\T^*M$. The Dolbeault Dirac operator 
$\overline{\partial}_{\tilde{L}}+\overline{\partial}_{\tilde{L}}^*$ maps 
$\Gamma(\Scal_J^\pm\otimes \tilde{L})$ to $\Gamma(\Scal_J^\mp\otimes\tilde{L})$, and the Riemann-Roch 
character $RR^K_J(M,\tilde{L})$ is defined as the equivariant index of the elliptic operator
$$
\overline{\partial}_{\tilde{L}}+\overline{\partial}_{\tilde{L}}^*:
\Gamma(\Scal_J^+\otimes \tilde{L})\longrightarrow \Gamma(\Scal_J^-\otimes\tilde{L})
$$

If $\epsilon_J=\pm 1$ is the quotient of the orientations defined by $\omega$ and by $J$, one has 
that 
$$
\Scal_P^\pm = \Scal_J^{\pm \epsilon_J}\otimes\tilde{L}.
$$
Hence $\Qcal^K_J(M)=\epsilon_J RR^K_J(M,\tilde{L})$ is defined as the equivariant index of the 
Dolbeault Dirac operator $\overline{\partial}_{\tilde{L}}+\overline{\partial}_{\tilde{L}}^*$ 
viewed as an elliptic operator $\Dcal^+_{\tilde{L}}$ from $\Gamma(\Scal_P^+)$ to $\Gamma(\Scal_P^-)$.

Finally we know that $\indice^K(\Dcal_P)=\indice^K(\Dcal^+_{\tilde{L}})$ since the first order elliptic 
operators $\Dcal_P$ and  $\Dcal^+_{\tilde{L}}$ have the same principal symbol \cite{Duistermaat96}. 
$\square$

\bigskip

In the remaining part of this paper, we find convenient to work with the following

\begin{defi}
A Hamiltonian $K$-manifold $(M,\omega,\Phi)$ is Spin-prequantized by $\tilde{L}$ if there exists 
an invariant almost complex structure $J$ {\em compatible} with $\omega$ such that 
$\tilde{L}^2\otimes\kappa_J^{-1}$ is a Kostant-Souriau line bundle over $(M,2\omega,2\Phi)$.
\end{defi}

We remark that $\esp_J=1$ when $J$ is {\em compatible} with $\omega$. Moreover, the Riemann-Roch character 
$RR^K_J(M,-)$ does not depend \cite{pep-RR} on the choice of the {\em compatible} invariant almost complex structure $J$ : we denote 
it simply by  $RR^K(M,-)$.

Finally, when a Hamiltonian manifold $(M,\omega,\Phi)$ is Spin-prequantized by the line bundle $\tilde{L}$, its  Spin-quantization is 
defined by 
$$
\spinq^K(M):=RR^K(M,\tilde{L}).
$$

\subsection{Functorial properties}

We summarize the functorial properties of $\spinq$ in the next 

\begin{prop}\label{prop-Qspin-functorial}
$\bullet$ If $(M,\omega,\Phi)$ is a Spin-prequantized Hamiltonian $K$-manifold, and $H$ is a closed subgroup 
of $K$ then the restriction of $\Qcal_{\hbox{{\rm \tiny spin}}}^K(M)$ to $H$ is equal to 
$\Qcal_{\hbox{{\rm \tiny spin}}}^H(M)$.

$\bullet$ If $(M_j,\omega_j,\Phi_j)$ are Spin-prequantized  Hamiltonian $K_j$-manifold, for $j=1,2$, then 
$M_1\times M_2$ is a Spin-prequantized  Hamiltonian $K_1\times K_2$-manifold and 
$$
 \Qcal_{\hbox{{\rm \tiny spin}}}^{K_1\times K_2}(M_1\times M_2)=\Qcal_{\hbox{{\rm \tiny spin}}}^{K_1}(M_1)\otimes 
\Qcal_{\hbox{{\rm \tiny spin}}}^{K_2}(M_2)
$$
in  $R(K_1\times K_2)\simeq R(K_1)\otimes R(K_2)$.

$\bullet$ If $(M,\omega_M,\Phi_M)$ and $(N,\omega_N,\Phi_N)$ are Spin-prequantized  Hamiltonian $K$-manifold, then 
$M\times N$ is a Spin-prequantized  Hamiltonian $K$-manifold and 
$$
\Qcal_{\hbox{{\rm \tiny spin}}}^{K}(M\times N)=\Qcal_{\hbox{{\rm \tiny spin}}}^{K}(M)\cdot 
\Qcal_{\hbox{{\rm \tiny spin}}}^{K}(N),
$$
where $\cdot$ denotes the product in  $R(K)$.

$\bullet$ A Spin-prequantization on $(M,\omega,\Phi)$ induces a Spin-prequantization on 
$\overline{M}:=(M,-\omega,-\Phi)$. The Spin-quantization of $\overline{M}$ corresponds to the 
dual of the Spin-quantization of $M$:
$$
\Qcal_{\hbox{{\rm \tiny spin}}}^{K}(\overline{M})=\left[\Qcal_{\hbox{{\rm \tiny spin}}}^{K}(M)\right]^*.
$$
\end{prop}

\medskip

\begin{proof} The first three points are direct consequences of the functorial properties of the 
index map. Let us prove the last point. One see that if $(\tilde{L},J)$ is a Spin-prequantum data 
for $M$ then $(\tilde{L}^{-1},-J)$ is a Spin-prequantum data for $\overline{M}$. Then 
we have for $X\in\kgot$ small enough
\begin{eqnarray*}
\spinq^K(\overline{M})(e^X)&=& \int_M e^{i(-\omega-\langle\Phi,X\rangle)}\widehat{A}(M)(X)\\
&=&\overline{ \int_M e^{i(\omega+\langle\Phi,X\rangle)}\widehat{A}(M)(X)}\qquad [1]\\
&=& \overline{\spinq^K(M)(e^X)}.\qquad [2]
\end{eqnarray*}
The relation $[1]$ is due to the fact that the differential form $\widehat{A}(M)(X)$ has real coefficients. 
Since $X\to \spinq^K(M)(e^X)$ are analytic functions, the identity $[2]$ shows that 
$\spinq^K(\overline{M})(k)=\overline{\spinq^K(M)(k)}$ for any $k\in K$. In other words the (virtual) 
representation $\spinq^K(\overline{M})$ corresponds to the dual of the  (virtual) 
representation $\spinq^K(M)$.
\end{proof}

\subsection{Spin-quantization of coadjoint orbits.}

Let $\mu\in\what{K}$ be a dominant weight. Let $\sigma$ be a face of the 
Weyl chamber such that $\mu\in\overline{\sigma}$: hence the stabilizer subgroup $K_\mu$ contains 
$K_\sigma$. We will restrict the one-dimensional representation $\C_\mu$ of $K_\mu$ to the subgroup 
$K_\sigma$.
 
Let $\rho$ be half the sum of the positive roots, and let $\rho_\sigma$ be half the 
sum of the positive roots which are orthogonal to $\sigma$. Note that $\rho-\rho_\sigma$ belongs to 
$\sigma$, hence $\mu+\rho-\rho_\sigma$ belongs also to $\sigma$ for any $\mu\in\overline{\sigma}$. 
The coadjoint orbit 
$$
\Ocal_\mu^\sigma:=K\cdot(\mu+\rho-\rho_\sigma)
$$
is $\spin$-prequantized by the compatible complex structure and the line bundle
$\tilde{L}= K\times_{K_\sigma}\C_{\mu}$. We have 
\begin{eqnarray*}
\spinq^K(\Ocal_\mu^\sigma)&=& RR^K(K/K_\sigma,K\times_{K_\sigma}\C_{\mu})\\
&=&  V^K_\mu.
\end{eqnarray*}
Thanks to Proposition \ref{prop-Qspin-functorial}, we know that 
$\spinq^K\left(\overline{\Ocal_\mu^\sigma}\right)= \left(V^K_\mu\right)^*$, 
where $\overline{\Ocal_\mu^\sigma}$ be the coadjoint orbit $\Ocal_\mu^\sigma$ with the 
opposite symplectic form. 

We have seen that the same irreducible representations $V_\mu^K$ can be realized as the 
Spin-quantization of the coadjoint orbits $\Ocal_\mu^\sigma$ where $\sigma$ is a face of the 
Weyl chamber containing $\mu$ in its closure.

\subsection{Spin-prequantization commutes with reduction}\label{subsec:spin-prequantization-reduction}

We consider first the case of a Hamiltonian $H$-manifold $(N,\omega,\Phi)$, not necessarily compact, which is 
Spin-prequantized by $\tilde{L}$. We suppose that $0$ is a regular value of $\Phi$. Let $N_0:=\Phi^{-1}(0)/H$ be the 
orbifold reduced space with its canonical symplectic structure $\omega_0$.

\begin{lem} \label{lem:restriction-spin-0}
The orbifold line bundle $\tilde{\Lcal}_0:=(\tilde{L}\vert_{\Phi^{-1}(0)})/H$ Spin-prequantizes $(N_0,\omega_0)$.
\end{lem}

\begin{proof} The fiber $\Zcal=\Phi^{-1}(0)$ is a smooth $H$-invariant submanifold of $N$. Let $\pi: \Zcal\to \Zcal/H= N_0$ be the projection.  
Recall that the symplectic structure $\omega_0$ on $N_0$ is defined by the relation $\pi^*(\omega_0)=\omega\vert_\Zcal$. 
Let $L_{2\omega}$ the Kostant-Souriau line bundle on $(N,2\omega,2\Phi)$ such that 
\begin{equation}\label{eq:lem-restriction}
\tilde{L}^2=L_{2\omega}\otimes \kappa_J.
\end{equation}
Here $J$ is a compatible invariant almost complex structure on $N$. We have $\T M\vert_\Zcal=\T \Zcal\oplus J(\hgot_\Zcal)$ where 
$\hgot_\Zcal\subset\T \Zcal$ is the trivial bundle given by the infinitesimal action of $H$. Since 
$\T \Zcal\simeq \pi^*(\T N_0)\oplus \hgot_\Zcal$ 
we get 
$$
\T M\vert_\Zcal\simeq \pi^*(\T N_0)\oplus \hgot_\Zcal\oplus J(\hgot_\Zcal).
$$
Hence $J$ induces a compatible almost complex structure $J_0$ on $(N_0,\omega_0)$, such that  
$(\kappa_J\vert_\Zcal)/H=\kappa_{J_0}$. 

The line bundle $L_{2\omega_0}= (L_{2\omega}\vert_{\Zcal})/H$ is a prequantum line bundle on $(N_0,\omega_0)$.
Finally, if we restrict (\ref{eq:lem-restriction}) to $\Zcal$, we get  
$$
\tilde{\Lcal}_0^2= L_{2\omega_0}\otimes\kappa_{J_0}.
$$
after taking the quotient by $H$. We have proved that $(J_0,\tilde{\Lcal}_0)$ Spin-prequantizes $(N_0,\omega_0)$.
\end{proof}

\medskip

For the rest of this section we consider a compact Hamiltonian $K$-manifold $(M,\omega,\Phi)$, that we suppose Spin-prequantized by 
the line bundle $\tilde{L}$.

 Let $\tau$ be a face of the Weyl chamber, and let $K_\tau$ be the commun stabilizer of points in $\tau$.
 Following Guillemin-Sternberg \cite{Guillemin-Sternberg84}, we introduce the following 
 $K_\tau$-invariant open subset of $\kgot^*_\tau$:
 $$
 U_\tau= K_\tau\cdot\{\xi\in\tgot^*_+\vert K_\xi\subset K_\tau\}= K_\tau\cdot \bigcup_{\tau\subset\overline{\sigma}}\sigma.
 $$
By construction, $U_\tau$ is a slice for the coadjoint action: this mean that the map $K\times U_\tau, (k,\xi)\mapsto k\cdot\xi$ 
factors through an inclusion $K\times_{K_\tau} U_\tau\croc \kgot^*$.

The symplectic cross-section theorem \cite{Guillemin-Sternberg84} asserts that the pre-image $Y_\tau=\Phi^{-1}(U_\tau)$ 
is a symplectic submanifold : we denote $\omega_\tau$ the restriction of $\omega$ to $Y_\tau$. The action 
of $K_\tau$ on $(Y_\tau,\omega_\tau)$ is Hamiltonian, where the restriction of $\Phi$ to $Y_\tau$ is a moment map. 
Since $\rho-\rho_\tau$ is a $K_\tau$-invariant element, we can use the translated moment map 
$\Phi_\tau: Y_\tau\to \kgot_\tau^*$ defined by 
$$
\Phi_\tau=\Phi\vert_{Y_\tau} - (\rho-\rho_\tau).
$$

\begin{lem}\label{lem:restriction-spin-slice}
The symplectic slice $(Y_\tau,\omega_\tau,\Phi_\tau)$ is Spin-prequantized by the line bundle 
$\tilde{L}_\tau:=\tilde{L}\vert_{Y_\tau}$.
\end{lem}

\begin{proof} 
We consider the open subset $K\times_{K_\tau}Y_\tau$ of $M$ and the projection  \break 
$\pi:K\times_{K_\tau}Y_\tau\to K/K_\tau$. We can suppose that the 
Spin-prequantum data, when restricted to $K\times_{K_\tau}Y_\tau$, is given by $(J,\tilde{L})$ where 
$J$ is a compatible almost complex structure on $K\times_{K_\tau}Y_\tau$ defined as the ``sum'' of 
the compatible almost complex structures $J_o$ and $J_\tau$: $J_o$ on $K/K_\tau$ and $J_\tau$ on $Y_\tau$. 

When we restrict the identity $L_{2\omega}=\tilde{L}^2\otimes \kappa_J^{-1}$ to $Y_\tau$ we get 
\begin{equation}\label{eq:L-tilde-Tau}
L_{2\omega}\vert_{Y_\tau}=(\tilde{L}\vert_{Y_\tau})^2\otimes \kappa_{J_\tau}^{-1}\otimes \C_{2(\rho-\rho_\tau)}.
\end{equation}
We consider the following line bundle on $Y_\tau$ :   
$$
L_{2\omega_\tau}:=L_{2\omega}\vert_{Y_\tau}\otimes\C_{2(\rho-\rho_\tau)}^{-1}.
$$
The relation (\ref{eq:L-tilde-Tau}) is then $L_{2\omega_\tau}=(\tilde{L}\vert_{Y_\tau})^2\otimes \kappa_{J_\tau}^{-1}$. Since 
$L_{2\omega_\tau}$ is a $K_\tau$-equivariant prequantum bundle over $(Y_\tau,2\omega_\tau,2\Phi_\tau)$, we 
conclude that $(Y_\tau,\omega_\tau,\Phi_\tau)$ is Spin-prequantized by the data $(J_\tau,\tilde{L}_\tau)$. 
\end{proof}

\medskip

Let us consider the case where $\tau=\sigma$ is {\em the smallest face} of the Weyl chamber so that 
moment polyhedron $\Delta(M):=\Phi(M)\cap\tgot^*_+$ is contained in the closure of $\sigma$. 
Then  the symplectic slice $Y_\sigma$ is equal to $\Phi^{-1}(\sigma)$, and the action of the subgroup 
$[K_\sigma,K_\sigma]$ is trivial on it \cite{L-M-T-W}.

We will then consider the Hamiltonian action of the center $Z_\sigma=Z(K_\sigma)$ 
on $Y_\sigma$. The map $\Phi_\sigma: Y_\sigma\to \kgot_\sigma^*$ takes values in 
$\zgot_\sigma^*=\R\sigma\subset\tgot^*$ and corresponds to the moment map 
relative to the action of $Z_\sigma$ on $(Y_\sigma,\omega_\sigma)$.  We know that 
$(Y_\sigma,\omega_\sigma,\Phi_\sigma)$ is Spin-prequantized by $\tilde{L}_\sigma:=\tilde{L}\vert_{Y_\sigma}$.

For each dominants weights $\mu$ which belongs to the closure of $\sigma$, we consider the 
symplectic reduction
\begin{eqnarray*}
M_\mu^\sigma&=&\Phi^{-1}(\Ocal_{\mu}^{\sigma})/K\\
&=& \Phi^{-1}_\sigma(\mu)/Z_\sigma.
\end{eqnarray*}

\medskip

{\bf For the rest of this section we fix a dominant weight $\mu\in\overline{\sigma}$  such that $\mu+\rho-\rho_\sigma\in\p(M)$, and 
we explain how one defines the Spin-quantization of the (possibly singular) reduced spaces $M_\mu^\sigma$.}

\medskip

Let $\pf\subset\zgot_\sigma^*$ be the rationnal vector subspace generated by $\{a-b\ \vert a,b\in\p(M)\}$.  
Let $\zgot_\sigma^\p\subset \zgot_\sigma$ be the subspace orthogonal to $\pf$, and let 
$Z_\sigma^\p\subset Z_\sigma$ be the corresponding subtorus. 

\begin{lem}\label{lem:Z-sigma-delta-action}
The torus $Z_\sigma^\p$ acts trivially on $Y_\sigma$ and on the line bundle $\tilde{L}_\sigma\otimes\C_{-\mu}$.
\end{lem}
\begin{proof} 
By definition of $\zgot_\sigma^\p$, $0=d\langle \Phi_\sigma, X\rangle=-\iota(X_{Y_\sigma})\omega_\sigma$ on 
$Y_\sigma$ for any $X\in\zgot_\sigma^\p$. Hence the torus $Z_\sigma^\p$ acts trivially on $Y_\sigma$. 
Let $L_{2\omega_\sigma}$ be the Kostant-Souriau line bundle over $(Y_\sigma,2\omega_\sigma,2\Phi_\sigma)$ so that 
$\tilde{L}_\sigma^2=L_{2\omega_\sigma}\otimes \kappa_{J_\sigma}$ (see Lemma \ref{lem:restriction-spin-slice}). 
We have on the section of $L_{2\omega_\sigma}$ the following equality of linear operators: 
$$
\Lcal(X)-\nabla_{X_M}= i \langle 2\Phi_\sigma, X\rangle,\quad \forall X\in\zgot_\sigma.
$$
If one takes $X\in\zgot_\sigma^\p$, the function $y\in Y_\sigma\mapsto \langle \Phi_\sigma(y), X\rangle$ 
is constant equal to $\langle \mu, X\rangle$. Finally 
$$
\Lcal(X)-2i\langle \mu, X\rangle=0,\quad \forall X\in\zgot_\sigma^\p
$$
as an operator on the section of $L_{2\omega_\sigma}$. In other words, the torus $Z_\sigma^\p$ acts trivially on
$L_{2\omega_\sigma}\otimes \C_{-2\mu}=(\tilde{L}_\sigma\otimes\C_{-\mu}) ^2\otimes\kappa_{J_\sigma}^{-1}$. 
Since $Z_\sigma^\p$ acts trivially on $\kappa_{J_\sigma}$, we conclude finally that $Z_\sigma^\p$ acts trivially on the 
line bundle $\tilde{L}_\sigma\otimes\C_{-\mu}$.
\end{proof}

Let $Z_\sigma'\subset Z_\sigma^\p$ be another subtorus such that $Z_\sigma=Z_\sigma^\p\times Z_\sigma'$: 
the dual of its Lie algebra $\zgot_\sigma'$ is identified with $\pf\subset\zgot_\sigma^*$. We look  now at 
$(Y_\sigma,\omega_\sigma)$ as a Hamiltonian $Z_\sigma'$-manifold with moment map 
$$
\Phi'_\sigma:=\Phi_\sigma-\mu=\Phi\vert_{Y_\sigma}-(\mu+\rho-\rho_\sigma).
$$
The $Z_\sigma'$-equivariant line bundle $\tilde{L}_\sigma':=\tilde{L}_\sigma\otimes\C_{-\mu}$ Spin-prequantizes
the Hamiltonian $Z_\sigma'$-manifold  $(Y_\sigma,\omega_\sigma,\Phi'_\sigma)$. 

If $0\in\pf$ is a regular value of $\Phi'_\sigma$, we know after Lemma \ref{lem:restriction-spin-0} that 
the orbifold reduced space $(M_\mu^\sigma,\omega_\mu^\sigma)$ is Spin-prequantized by the line bundle 
$$
\tilde{\Lcal}_{\mu}^\sigma:=\left(\tilde{L}\vert_{\Phi_\sigma^{-1}(\mu)}\otimes\C_{-\mu}\right)/Z'_\sigma, 
$$
and its Spin-quantization $\Qcal_{\hbox{\rm\tiny\spin}}(M_\mu^\sigma)$ 
is defined like in Definition \ref{def:Q-spin}. In the general case where $0\in\pf$ is not necessarily a regular 
value of $\Phi'_\sigma$ we proceed by shift desingularization. For $\esp\in\pf$ small enough and 
generic we consider the orbifold reduced space 
$$
M_{\mu+\esp}^\sigma:=(\Phi'_\sigma)^{-1}(\esp)/Z_\sigma'= \Phi_\sigma^{-1}(\mu+\esp)/Z'_\sigma 
$$
and its orbifold line bundle
$$
\tilde{\Lcal}_{\mu+\esp}^\sigma:=\left(\tilde{L}\vert_{\Phi_\sigma^{-1}(\mu+\esp)}\otimes\C_{-\mu}\right)/Z'_\sigma.
$$

The following crucial fact is proved in Section \ref{sec:torus-case}.

\begin{theo} \label{theo:quant-M-sigma-mu}
The Riemann-Roch number $RR(M_{\mu+\esp}^\sigma, \tilde{\Lcal}_{\mu+\esp}^\sigma)\in\Z$ does not depend 
of the choice of a generic  and small enough $\esp\in\pf$.
\end{theo}

Thanks to the last Theorem we can define the quantization $\spinq(M_\mu^\sigma)\in\Z$ of the 
(possibly singular) reduced space $M_\mu^\sigma$ for $\mu\in\what{K}\cap\overline{\sigma}$.

\begin{defi}\label{def:indice-reduction}
Let $\mu\in\what{K}\cap\overline{\sigma}$. 

$\bullet$ If $\mu +\rho-\rho_\sigma\in\Delta(M)$, the integer 
$\Qcal_{\hbox{{\rm \tiny spin}}}(M_\mu^\sigma)\in\Z$ is defined as the Riemann-Roch character 
$RR(M_{\mu+\esp}^\sigma, \tilde{\Lcal}_{\mu+\esp}^\sigma)$ for $\esp\in\pf$ generic and small enough.

$\bullet$  If $\mu +\rho-\rho_\sigma\notin\Delta(M)$, we set $\Qcal_{\hbox{{\rm \tiny spin}}}(M_\mu^\sigma)=0$
\end{defi}

\begin{rem}
If $\mu +\rho-\rho_\sigma$ does not belongs to the relative interior  of $\p(M)$, we can choose $\esp$ so that 
$\mu +\rho-\rho_\sigma+\esp\notin\p(M)$. Then the reduced space $M_{\mu+\esp}^\sigma$ is empty and 
the corresponding Riemann-Roch character $RR(M_{\mu+\esp}^\sigma, \tilde{\Lcal}_{\mu+\esp}^\sigma)$ vanishes. 
Hence $\Qcal_{\hbox{{\rm \tiny spin}}}(M_\mu^\sigma)=0$. 
\end{rem}

\section{Spin-quantization commutes with reduction}

Let $(M,\omega,\Phi)$ be a compact Hamiltonian $K$-manifold which is $\spin$ prequantized. 
We are looking to a geometric interpretation of the multiplicity, denoted $\mm_\mu$, of the 
representation $V_\mu^K$ into $\spinq^K(M)$.

The main result of this paper is the following

\begin{theo}\label{theo:QR-spin}
 Let $\sigma$ be the smallest face of the Weyl chamber so that 
 $\Phi(M)\cap\tgot^*_+\subset\overline{\sigma}$. For $\mu\in\what{K}$, we have
 $$
 \mm_\mu=\begin{cases}
   0 & \text{if}\quad \mu\notin\overline{\sigma};\\
    \Qcal_{\hbox{{\rm \tiny spin}}}(M_\mu^\sigma) & \text{if}\quad
   \mu\in\overline{\sigma}.
\end{cases}
$$

\end{theo}

In this section we introduce the main tools needed for the proof of Theorem \ref{theo:QR-spin}. 

In Section \ref{subsec:trans-elliptic}, we recall the notion of {\em tranversally elliptic symbols}.

In Section \ref{subsec:witten-deformation}, we recall the Witten's way of localization the 
Riemann-Roch character \cite{pep-RR}. We recall in Proposition \ref{prop:RR-beta=0}, the criterium 
observed in \cite{pep-ENS} for the vanishing of the invariant part of the localized Riemann-Roch character.

In Section \ref{subsec:induction-formula}, we recall an induction formula proved in 
\cite{pep-RR,pep-ENS} for the localized Riemann-Roch character. 

In Section \ref{sec:torus-case}, we prove Theorem \ref{theo:QR-spin} when $K$ is a torus\footnote{This situation was 
already handeld in \cite{pep-ENS}.}. We give by the same way a proof of Theorem \ref{theo:quant-M-sigma-mu} which 
is essential to the definition of  the Spin-quantization of the (possibly singular) reduced spaces $M_\mu^\sigma$.

\subsection{Elliptic and transversally elliptic symbols}\label{subsec:trans-elliptic}

Here we give the basic definitions from the theory of transversally
elliptic symbols (or operators) defined by Atiyah-Singer in
\cite{Atiyah74}. For an axiomatic treatment of the index morphism
see Berline-Vergne \cite{B-V.inventiones.96.1,B-V.inventiones.96.2} and 
Paradan-Vergne \cite{pep-vergne:bismut}. For a short introduction see \cite{pep-RR}.

Let $\Xcal$ be a {\em compact} $K$-manifold. Let $p:\T
\Xcal\to \Xcal$ be the projection, and let $(-,-)_\Xcal$ be a
$K$-invariant Riemannian metric. If $E^{0},E^{1}$ are
$K$-equivariant complex vector bundles over $\Xcal$, a
$K$-equivariant morphism $h \in \Gamma(\T
\Xcal,\hom(p^{*}E^{0},p^{*}E^{1}))$ is called a {\em symbol} on $\Xcal$. The
subset of all $(x,v)\in \T \Xcal$ where\footnote{The map $h(x,v)$ will be also denote 
$h\vert_x(v)$} $h(x,v): E^{0}_{x}\to
E^{1}_{x}$ is not invertible is called the {\em characteristic set}
of $h$, and is denoted by $\Char(h)$.

In the following, the ``product'' of a symbol  $h$ 
by a complex vector bundle $F\to M$, is the symbol
$$
h\otimes F
$$
defined by $h\otimes F(x,v)=h(x,v)\otimes {\rm Id}_{F_x}$ from 
$E^{0}_x\otimes F_x$ to $E^{1}_x\otimes F_x$. Note that $\Char(h\otimes F)=\Char(h)$.

Let $\T_{K}\Xcal$ be the following subset of $\T \Xcal$ :
$$
   \T_{K}\Xcal\ = \left\{(x,v)\in \T \Xcal,\ (v,X_{\Xcal}(x))_{_{\Xcal}}=0 \quad {\rm for\ all}\
   X\in\kgot \right\} .
$$

A symbol $h$ is {\em elliptic} if $h$ is invertible
outside a compact subset of $\T \Xcal$ (i.e. $\Char(h)$ is
compact), and is $K$-{\em transversally elliptic} if the
restriction of $h$ to $\T_{K}\Xcal$ is invertible outside a
compact subset of $\T_{K}\Xcal$ (i.e. $\Char(h)\cap
\T_{K_2}\Xcal$ is compact). An elliptic symbol $h$ defines an
element in the equivariant $\K$-theory of $\T\Xcal$ with compact
support, which is denoted by $\K_{K}(\T \Xcal)$, and the
index of $h$ is a virtual finite dimensional representation of
$K$, that we denote $\indice^K_{\Xcal}(h)\in R(K)$
\cite{Atiyah-Segal68,Atiyah-Singer-1,Atiyah-Singer-2,Atiyah-Singer-3}.

Let  
$$
\Rforc(K)\subset R^{-\infty}(K)
$$ 
be the $R(K)$-submodule formed by all the infinite sum $\sum_{\mu\in\what{K}}m_\mu V_\mu^K$ 
where the map $\mu\in\what{K}\mapsto m_\mu\in\Z$ has at most a {\em
polynomial} growth. The $R(K)$-module $\Rforc(K)$ is the Grothendieck group 
associated to the {\em trace class} virtual $K$-representations: we can associate 
to any $V\in\Rforc(K)$, its trace $k\to {\rm Tr}(k,V)$ which is a generalized function on $K$ 
invariant by conjugation. Then the trace defines a morphism of $R(K)$-module
\begin{equation}\label{eq:trace}
\Rforc(K)\croc \fgene(K)^K.
\end{equation}

A $K$-{\em transversally elliptic} symbol $h$ defines an
element of $\K_{K}(\T_{K}\Xcal)$, and the index of
$h$ is defined as a trace class virtual representation of $K$, that we still denote 
$\indice^K_\Xcal(h)\in \Rforc(K)$. 

Remark that any elliptic symbol of $\T \Xcal$ is $K$-transversally
elliptic, hence we have a restriction map $\K_{K}(\T
\Xcal)\to \K_{K}(\T_{K}\Xcal)$, and a commutative
diagram
\begin{equation}\label{indice.generalise}
\xymatrix{ \K_{K}(\T \Xcal) \ar[r]\ar[d]_{\indice^K_\Xcal}
&
\K_{K}(\T_{K}\Xcal)\ar[d]^{\indice^K_\Xcal}\\
R(K)\ar[r] & \Rforc(K)\ .
   }
\end{equation}

\medskip

Using the {\em excision property}, one can easily show that the
index map $\indice^K_\Ucal: \K_{K}(\T_{K}\Ucal)\to
\Rforc(K)$ is still defined when $\Ucal$ is a
$K$-invariant relatively compact open subset of a
$K$-manifold (see \cite{pep-RR}[section 3.1]). 

\medskip 

Suppose that $M$ is a $K$-manifold equipped with an invariant almost complex structure $J$. 
Let us recall the definition of the Riemann-Roch character $RR^K_J(M,-)$. 

The complex vector bundle $(\T^* M)^{0,1}$ is $K$-equivariantly identified 
with the tangent bundle $\T M$ equipped with the complex structure $J$.
We work with the  Hermitian structure on  $(\T M,J)$ defined by : 
$(v,w):=\Omega(v,J w) - i \Omega(v,w)$ for $v,w\in \T M$. The symbol 
$$
\Thom(M,J)\in 
\Gamma\left(M,\hom(p^{*}(\wedge_{\C}^{even} \T M),\,p^{*}
(\wedge_{\C}^{odd} \T M))\right)
$$  
at $(m,v)\in \T M$ is equal to the Clifford map
\begin{equation}\label{eq.thom.complex}
 \clif_{m}(v)\ :\ \wedge_{\C}^{even} \T_m M
\longrightarrow \wedge_{\C}^{odd} \T_m M,
\end{equation}
where $\clif_{m}(v).w= v\wedge w - \iota(v)w$ for $w\in 
\wedge_{\C}^{\bullet} \T_{m}M$. Here $\iota(v):\wedge_{\C}^{\bullet} 
\T_{m}M\to\wedge^{\bullet -1} \T_{m}M$ denotes the 
contraction map. Since $\clif_{m}(v)^2=-\| v\|^2 {\rm Id}$, the map  
$\clif_{m}(v)$ is invertible for all $v\neq 0$. Hence the characteristic set 
of $\Thom(M,J)$ corresponds to the $0$-section of $\T M$.

Let $E$ be a $K$-equivariant complex vector bundle over $M$. It is a classical fact that the principal symbol  of 
the Dolbeault-Dirac operator $\overline{\partial}_{E}+\overline{\partial}_{E}^*$ is equal to the following 
elliptic symbol\footnote{Here we use an identification $\T^*M\simeq \T M$ given by an invariant Riemannian metric.}  
$$
\clif_E:=\Thom(M,J)\otimes E,
$$
see \cite{Duistermaat96}.  Since $M$ is compact, the symbol $\clif_E$ is elliptic and
then defines an element of the equivariant \textbf{K}-group of
$\T M$. 
\begin{defi}\label{def:RR}
The Riemann-Roch character $RR^K_J(M,E)\in R(K)$ is defined equivalently 

$\bullet$ as the topological index of $\clif_E\in\K_{K}(\T M)$, or 

$\bullet$ as the analytical index of the
Dolbeault-Dirac operator $\overline{\partial}_E+\overline{\partial}_E^*$.
\end{defi}

\subsection{Localization of the Riemann-Roch character}\label{subsec:witten-deformation}

Let $(M,\omega,\Phi)$ a compact Hamiltonian $K$-manifold Spin-prequantized by  
$(\tilde{L}, J)$ where $J$ is a {\em compatible} almost complex structure on $M$. 
The Riemann-Roch character attached to $J$ is just denoted $RR^K(M,-)$.

By definition the Spin-quantization of $(M,\omega,\Phi)$ is 
$$
\spinq^K(M):= RR^K(M,\tilde{L})\ \in\ R(K).
$$

We recall the Witten's deformation of the Riemann-Roch character \cite{pep-RR, pep-ENS}. We use in all 
this paper an isomorphism $\kgot^*\simeq \kgot$ defined by a $K$-invariant scalar product
on  $\kgot^*$. In order to simplify the notation, we use the same symbol for $\xi\in\kgot^*$ and its 
corresponding element in $\kgot$.

The moment map $\Phi$ is seen as en equivariant map from $M$ to $\kgot$. We define the {\em Kirwan vector field} on
$M$ : 
\begin{equation}\label{eq-kappa}
    \kappa_m= \left(\Phi(m)\right)_M(m), \quad m\in M.
\end{equation}

\begin{defi}\label{def:pushed-sigma}
The symbol  $\clif_{\tilde{L}}=\Thom(M,J)\otimes \tilde{L}$ pushed by the vector field $\kappa$ is the 
symbol $\clif^\kappa_{\tilde{L}}$ 
defined by the relation
$$
\clif^\kappa_{\tilde{L}}\vert_m(v)=\Thom(M,J)\otimes \tilde{L}\vert_m(v-\kappa_m)
$$
for any $(m,v)\in\T M$.
\end{defi}

Note that $\clif^\kappa_{\tilde{L}}\vert_m(v)$ is invertible except if
$v=\kappa_m$. If furthermore $v$ belongs to the subset $\T_K M$
of tangent vectors orthogonal to the $K$-orbits, then $v=0$ and
$\kappa_m=0$.  Indeed $\kappa_m$ is tangent to $K\cdot m$ while
$v$ is orthogonal.

Since $\kappa$ is the Hamiltonian vector field of the function
$\frac{-1}{2}\|\Phi\|^2$, the set of zeros of $\kappa$ coincides with the set
$\Cr(\|\Phi\|^2)$ of critical points of $\|\Phi\|^2$. Finally we have 
\begin{eqnarray*}
\Char(\clif^\kappa_{\tilde{L}})\cap \T_K M &\simeq&
\Cr(\|\Phi\|^2)\\
&=& \bigcup_{\beta\in\Bcal} \underbrace{K\cdot\left(M^\beta\cap\Phi^{-1}(\beta)\right)}_{C_\beta}
\end{eqnarray*}
where $\Bcal\subset \kgot^*$ is a finite subset parametrizing coadjoint orbits $K\cdot\beta$.

We are interested to the restriction $\clif^\kappa_{\tilde{L}}\vert_{U}$ of the elliptic symbol on an invariant open  subset 
$U\subset M$. Note that the set  $\Char(\clif^\kappa_{\tilde{L}}\vert_{U})\cap \T_K U \simeq \Cr(\|\Phi\|^2)\cap U$ is compact when  
\begin{equation}\label{eq:condition-U}
\partial \Ucal\cap \Cr(\|\Phi\|^2)=\emptyset.
\end{equation}
When (\ref{eq:condition-U}) holds we denote 
\begin{equation}\label{eq:Q-Phi-U}
\Qcal_{\Phi}^K(U):= \indice^K_{U}(\clif^\kappa_{\tilde{L}}\vert_{U})\quad
   \in\quad \Rforc(K)
\end{equation}
the equivariant index of the transversally elliptic symbol $\clif^\kappa_{\tilde{L}}\vert_{U}$.

For any $\beta\in \Bcal$, we consider a relatively compact open 
invariant neighborhood 
$U_\beta$ of  $C_\beta$ such that 
$\Cr(\|\Phi\|^2)\cap \overline{U_\beta}= C_\beta$. 

\begin{defi}\label{def:Q-beta}
We denote 
$$
\Qcal^K_{\beta}(M)\in\Rforc(K)
$$
the index of the transversally elliptic symbol $\clif^\kappa_{\tilde{L}}\vert_{U_\beta}$.
\end{defi}

Everything can be defined if we replace the line bundle $\tilde{L}$ by any equivariant complex vector bundle $E$. 
We can consider the pushed symbol $\clif_{E}^\kappa$,  and the localized Riemann-Roch characters 
$$
RR^K_\Phi(U,E):=\indice^K_{U}(\clif^\kappa_E\vert_{U})\quad {\rm and}\quad  
RR^K_\beta(M,E):=\indice^K_{U_\beta}(\clif^\kappa_E\vert_{U_\beta}).
$$

\medskip

A direct application of the excision property \cite{pep-RR} gives that 
\begin{equation}\label{eq:witten-loc}
\spinq^K(M)=\sum_{\beta\in\Bcal}\Qcal^K_\beta(M).
\end{equation}
If we work with $RR^K_\Phi(U,E)$, we have 
\begin{equation}\label{eq:witten-loc-RR}
RR^K_\Phi(U,E)=\sum_{\beta\in \Bcal\cap\Phi(U)}RR^K_\beta(U,E).
\end{equation}

The decomposition (\ref{eq:witten-loc}) and (\ref{eq:witten-loc-RR}) will be used in the next chapters 
when one want to compute the multiplicity, denoted $[\spinq^K(M)]^K$, of 
the trivial representation in $\spinq^K(M)$. We have 
$$
\left[\spinq^K(M)\right]^K=\sum_{\beta\in\Bcal}\left[\Qcal^K_\beta(M)\right]^K.
$$
and we finish this section by recalling a {\em criterium} under which one has $[\Qcal^K_\beta(M)]^K=0$. 

\medskip

Let $\beta$ be a non-zero element in $\kgot$: let $\tore_\beta\subset K$ be the torus generated 
by $\beta$.  For $m\in M^{\beta}$, let $\alpha^m_1,\cdots, \alpha^m_p$ 
be the real infinitesimal weights for the action of $\tore_{\beta}$ on the fibers of 
$\T_m M$ (we equip the fibers of $\T_m M/\T_m M^\beta$ with a $\tore_\beta$-invariant complex structure).  

\begin{defi}\label{trace-beta} 
    Let us denote by $\tr_{\beta}|\T_m M|$ the following 
positive number
$$
\tr_{\beta}|\T_m M|:=\sum_{i=1}^l |\langle
\alpha^m_i,\beta\rangle|\ .
$$
\end{defi} 

Note that $m\in M^\beta\mapsto \tr_{\beta}|\T_m M|$ is constant along a connected component of $M^\beta$. 
We see also that the expression $\tr_{\beta}|E|$ is well defined for any $H$-equivariant real 
vector bundle $E\to P$, when $\beta\in\hgot$ acts trivially on $P$. 

\begin{exam}\label{example:Tr-beta}
The map  $\beta\in\kgot\mapsto\tr_{\beta}|\kgot|$ is invariant under the adjoint action. When $\beta$ belongs to the 
Weyl chamber, one has $\tr_{\beta}|\kgot|=2(\rho,\beta)$. Note that $\tr_{\beta}|\kgot|\leq 2\|\rho\|\,\|\beta\|$ for any 
$\beta\in\kgot$.
\end{exam}

We have proved in \cite{pep-ENS} the following nice criterium.

\begin{prop} \label{prop:RR-beta=0}
Let $\beta\neq 0$ in $\Bcal$. The multiplicity of the trivial representation in
$\Qcal^K_\beta(M)$ is equal to zero if
\begin{equation}\label{eq.RR-beta=0}
\parallel\beta\parallel^2 +
\frac{1}{2}\tr_{\beta}|\T_m M|>\tr_{\beta}|\kgot|,\qquad \forall\ m\in M^\beta\cap\Phi^{-1}(\beta).
\end{equation}
\end{prop}

\begin{rem}
Note that condition (\ref{eq.RR-beta=0}) is equivalent to 
\begin{equation}\label{eq.RR-beta=0.bis}
\parallel\Phi(m)\parallel^2 +
\frac{1}{2}\tr_{\Phi(m)}|\T_{m} M|>\tr_{\Phi(m)}|\kgot|,\qquad \forall\ m\in C_\beta.
\end{equation}
\end{rem}

\medskip

If the critical set $C_\beta$ decomposes in a finite disjoint union of closed $K$-invariant subset 
$C_\beta=\cup_j C_\beta^j$, we consider invariant open neighborhood $U^j$ of $C^j_\beta$ such that 
$\overline{U^j_\beta}\cap\Cr(\|\Phi\|^2)=C^j_\beta$, and we define 
$$
\Qcal_{C_\beta^j}^K(M):=\indice^K_{U^j_\beta}(\clif^\kappa\vert_{U^j_\beta})\quad
   \in\quad \Rforc(K)
 $$
Then the generalized character $\Qcal^K_\beta(M)$ is equal to the sum 
$\sum_j\Qcal^K_{C_\beta^j}(M)$ and Proposition \ref{prop:RR-beta=0} tells us 
that $[\Qcal^K_{C_\beta^j}(M)]^K=0$ if (\ref{eq.RR-beta=0.bis}) holds on $C^j_\beta$.

\subsection{Induction formulas}\label{subsec:induction-formula}

Let $H$ be a compact connected Lie group. Let $H\cdot a$ be a coadjoint orbit. Let $(N,\omega_N,\Phi_N)$ be an 
Hamiltonian $H$-manifold which is not assumed to be compact. But we assume that $\Phi_N$ is {\em proper} near 
$H\cdot a$: the pullback $\Phi^{-1}_N(\Ccal)$ is compact if $\Ccal\subset\hgot^*$ is a small enough compact 
invariant  neighborhood of $H\cdot a$.

Let $H_a$ be the stabilizer of $a\in\hgot^*$, and let $Y_a$ be a symplectic slice near $H\cdot a$: $Y_a$ is 
a $H_a$-invariant symplectic manifold of $N$ such that $\Phi_N(Y_a)\subset\hgot_a^*$ and such that $H\times_{H_a} Y_a$ 
is  diffeomorphic to an invariant open neighborhood of $\Phi^{-1}_N(H\cdot a)$. We will work with the following moment map on $Y_a$:
$$
\Phi_{Y_a}=\Phi_N\vert_{Y_a}-a.
$$

Let $N\times \overline{H\cdot a}$ be the Hamiltonian $H$-manifold, with moment map $\Phi(n,\xi)=\Phi_N(n)-\xi$. Let 
$$
RR^H_0(N\times \overline{H\cdot a},\ -\ )
$$
be the Riemann-Roch character localized near the compact subset $\Phi^{-1}(0)\subset N\times \overline{H\cdot a}$. 
Let 
$$
RR^{H_a}_0(Y_a,\ -\ )
$$
be the Riemann-Roch character localized near the compact subset $\Phi_{Y_a}^{-1}(0)=\Phi^{-1}_N(a)\subset Y_a$.

Let $\indHa: \Rfor(H_a)\to \Rfor(H)$ be the induction map. If $E$ and $F$ are respectively $H$-equivariant complex 
vector bundles on $N$ and $H\cdot a$, we denote $E\boxtimes F$ their product. We have proved in \cite{pep-RR} 
(see also Proposition 4.13 in \cite{pep-ENS}) the following induction formula

\begin{prop}\label{prop:induction} For any equivariant complex vector bundles $E\to N$ and $F\to H\cdot a$, we have 
$$
RR^H_0(N\times \overline{H\cdot a},E\boxtimes F)
= \indHa\left[ 
RR^{H_a}_0(Y_a,E\vert_{Y_a}\otimes F\vert_{\{a\}} )\right].
$$
\end{prop} 

The last Proposition gives in particular  that
\begin{equation}\label{eq:induction}
\left[RR^H_0(N\times \overline{H\cdot a},E\boxtimes F)\right]^H
= \left[ 
RR^{H_a}_0(Y_a,E\vert_{Y_a}\otimes F\vert_{\{a\}} )\right]^{H_a}.
\end{equation}

\subsection{The torus case}\label{sec:torus-case}

Let $T$ be a compact torus, and let $(M,\omega,\Phi)$ be a compact Hamiltonian $T$-manifold 
which is Spin-prequantized by the data $(J,\tilde{L})$. We suppose that $J$ is compatible with $\omega$. 
The irreducible representation of $T$ is parametrized by the lattice $\what{T}\subset\tgot^*$: at each $\mu\in\what{T}$ we associate 
the one-dimensional representation $\C_\mu$. 

We write $\spinq^T(M)=\sum_{\mu\in\what{T}} \mm_\mu \C_\mu$, 
and one wants to show that the multiplicity $\mm_\mu$ is equal to the Spin-quantization of the 
(possibly singular) reduced space $M_\mu:=\Phi^{-1}(\mu)/T$.

We fix once for all $\mu\in\what{T}$. And we apply the Witten deformation procedure 
to the Hamiltonian $T$-manifold  $(M,\omega,\Phi-\mu)$ which is Spin-prequantized by $(J,\tilde{L}\otimes\C_{-\mu})$. 
We have 
$$
\mm_\mu=\sum_{\beta\in\Bcal^\mu}\left[RR^T_\beta(M,\tilde{L}\otimes\C_{-\mu})\right]^T
$$
where $\Bcal^\mu$ parametrizes the critical points of $\|\Phi-\mu\|^2$. Here the criterion (\ref{eq.RR-beta=0}) holds for 
any non-zero $\beta$ since the Lie algebra $\tgot$ is abelian. We have then
$$
\mm_\mu=\left[RR^T_0(M,\tilde{L}\otimes\C_{-\mu})\right]^T.
$$
In particular  $\mm_\mu=0$ if $\mu\notin \Phi(M)$. When $\mu\in\Phi(M)$, we consider a small neighborhood $U$ of 
$\Phi^{-1}(\mu)\subset M$ so that $\overline{U}\cap \Cr(\|\Phi-\mu\|^2)= \Phi^{-1}(\mu)$. We know then that 
\begin{equation}\label{eq:RR-loc-U}
\mm_\mu=\left[RR^{T}_{\Phi-\mu}(U,\tilde{L}\vert_U\otimes\C_{-\mu})\right]^{T}.
\end{equation}

\subsubsection{First case: $\mu$ is a regular value of $\Phi$} 

We consider the orbifold reduced space $M_{\mu}=\Phi^{-1}(\mu)/T$ which is 
equipped with a canonical symplectic form $\omega_\mu$. Let $RR(M_{\mu},-)$ be the Riemann-Roch character attached to 
a compatible almost complex struture. We prove in \cite{pep-RR} that for any complex vector bundle $E\to U$
\begin{equation}\label{eq:RR-reduit}
\left[RR^{T}_{\Phi-\mu}(U,E)\right]^{T}=RR(M_{\mu},\Ecal)
\end{equation}
where $\Ecal= E\vert_{\Phi^{-1}(\mu)}/T$ is the induced orbifold bundle on $M_{\mu}$. 
If we take $E=\tilde{L}\vert_U\otimes\C_{-\mu}$ on sees (thanks to Lemma \ref{lem:restriction-spin-0}) that 
$$
\tilde{\Lcal}_\mu=(\tilde{L}\vert_{\Phi^{-1}(\mu)}\otimes\C_{-\mu})/T
$$
is an orbifold line bundle which Spin-prequantizes $(M_\mu,\omega_\mu)$, and (\ref{eq:RR-reduit}) gives together 
with (\ref{eq:RR-loc-U}) that
$$
\mm_\mu=RR(M_{\mu},\tilde{\Lcal}_\mu)=\spinq(M_\mu).
$$

\subsubsection{Second case : $\mu$ is a not (necessarilly) a regular value of $\Phi$} 

Let $\pf\subset\tgot^*$ be the rationnal vector subspace generated by $\{a-b\ \vert a,b\in\Phi(M)\}$. 
We work here with a weight $\mu\in\Phi(M)$ so that the polytope $\Phi(M)$ lives in the affine subspace $\mu+\pf$.
Let $\tgot_\Delta\subset \tgot$ be the subspace orthogonal to $\Delta$, and let $T_\Delta\subset T$ be the corresponding subtorus.

\begin{lem}
The group $T_\p$ acts trivially on $M$ and on the line bundle $\tilde{L}\otimes\C_{-\mu}$.
\end{lem}
\begin{proof} 
See the proof of Lemma \ref{lem:Z-sigma-delta-action}.
\end{proof}

Let $T'\subset T$ be another subtorus such that $T=T_\p\times T'$: the dual of its Lie algebra $\tgot'$ is identified with 
$\pf\subset\tgot^*$. We look  now at $(M,\omega)$ as a Hamiltonian $T'$-manifold with moment map 
$$
\Phi':=\Phi-\mu: M\longrightarrow \pf=(\tgot')^*
$$
The $T'$-equivariant line bundle $\tilde{L}':=\tilde{L}\otimes\C_{-\mu}$ Spin-prequantizes the Hamiltonian $T'$-manifold  
$(M,\omega,\Phi')$. Let $U$ be a small neighborhood of $\Phi'^{-1}(0)$ in $M$. The generalized character 
$RR^{T}_{\Phi-\mu}(U,\tilde{L}\vert_U\otimes \C_{-\mu})$ belongs to $\Rfor(T')$ 
and corresponds to the localized Riemann-Roch character 
$$
RR^{T'}_{\Phi'}(U,\tilde{L}'\vert_U).
$$ 

We deform the moment map $\Phi'$ in $\Phi'-\esp$ where $\esp$ is a small element in $\pf$. We have proved in 
\cite{pep-ENS}[Proposition 4.14] the following 

\begin{lem}\label{lem-1}
$\bullet$  If $\esp$ is small enough, the critical set of $\|\Phi'-\esp\|^2$ does not intersect $\partial U$, so that 
the localized Riemann-Roch character $RR^{T'}_{\Phi'-\esp}(U,-)$ is well defined.

$\bullet$ We have $RR^{T'}_{\Phi'}(U,\tilde{L}'\vert_U)=RR^{T'}_{\Phi'-\esp}(U,\tilde{L}'\vert_U)$ if $\esp$ is small enough. 
\end{lem}

Now we are left to the computation of $\mm_\mu=\left[RR^{T'}_{\Phi'-\esp}(U,\tilde{L}'\vert_U)\right]^{T'}$ when  
$\esp\in\pf$ is small enough. 
We start with the decomposition 
$$
RR^{T'}_{\Phi'-\esp}(U,\tilde{L}'\vert_U)=\sum_{\beta\in\Bcal_\esp}RR^{T'}_{\Phi'-\esp,\beta}(U,\tilde{L}'\vert_U)
$$ 
where $RR^{T'}_{\Phi'-\esp,\beta}(U,-)$ denotes the Riemann-Roch charcater localized near 
the compact subset $U^\beta\cap(\Phi')^{-1}(\beta+\esp)$.  We have proved in \cite{pep-ENS}[Lemma 4.16] the following

\begin{lem}\label{lem-2} 
If $\esp$ is small enough we have $\left[RR^{T'}_{\Phi'-\esp,\beta}(U,\tilde{L}'\vert_U)\right]^{T'}=0$ 
when $\beta\neq 0$.
\end{lem}

At $\esp\in\pf$ small enough and generic we associate the orbifold $M_{\mu+\esp}=$ \break $\Phi^{-1}(\mu+\esp)/T'$ which 
is equipped with the orbifold line bundle 
$$
\tilde{\Lcal}_{\mu+\esp}= \left(\tilde{L}\vert_{\Phi^{-1}(\mu+\esp)}\otimes\C_{-\mu}\right)/T'.
$$
Let $RR(M_{\mu+\esp},-)$ be the Riemann-Roch map associated to a compatible almost complex structure. 
If we use (\ref{eq:RR-reduit}) together with the Lemmas \ref{lem-1} and \ref{lem-2} we get 

\begin{theo}\label{theo:QR-abelien} 
The multiplicity $\mm_\mu$ is equal to the Riemann-Roch number \break 
$RR(M_{\mu+\esp},\tilde{\Lcal}_{\mu+\esp})\in\Z$ where $\esp\in\pf$ is small and generic. 
\end{theo}

We prove here that the quantity $RR(M_{\mu+\esp}, \tilde{\Lcal}_{\mu+\esp})$ does not depend of the 
choice of $\esp$ small and generic: it is the definition of the Spin quantization, denoted $\spinq(M_\mu)$, of the 
(possibly singular) reduced space $M_\mu$.

\subsubsection{Proof of Theorem \ref{theo:quant-M-sigma-mu}}

The same kind of proof work for Theorem \ref{theo:quant-M-sigma-mu}. We consider an invariant  relatively compact neighborhood 
$U_{\sigma,\mu}$ of $\Phi_\sigma^{-1}(\mu)=$ \break $\Phi^{-1}(\mu+\rho-\rho_\sigma)$ in the slice $Y_\sigma$ so that 
$\Cr(\|\Phi_\sigma-\mu \|^2)\cap \overline{U_{\sigma,\mu}}=\Phi_\sigma^{-1}(\mu)$. Thanks to Lemmas \ref{lem-1} and \ref{lem-2}, 
we know that the Riemann-Roch character 
$$
RR^{Z_\sigma'}_{\Phi_\sigma-\mu-\esp}(U_{\sigma,\mu},\tilde{L}')\in\Rfor(Z_\sigma')
$$
are well defined for $\esp\in\pf$ small enough, and they do not depend of the choice of $\esp$. If $\esp_1,\esp_2\in\pf$ are small 
enough regular values of $\Phi_\sigma-\mu$ we get thanks to (\ref{eq:RR-reduit}) that 
\begin{eqnarray*}
RR(M_{\mu+\esp_1}^\sigma,\tilde{\Lcal}^\sigma_{\mu+\esp_1})&=&
\left[RR^{Z_\sigma'}_{\Phi_\sigma-\mu-\esp_1}(U_{\sigma,\mu},\tilde{L}')\right]^{Z_\sigma'}\\
&=&
\left[RR^{Z_\sigma'}_{\Phi_\sigma-\mu-\esp_2}(U_{\sigma,\mu},\tilde{L}')\right]^{Z_\sigma'}\\
&=& 
RR(M_{\mu+\esp_2}^\sigma,\tilde{\Lcal}^\sigma_{\mu+\esp_2}).
\end{eqnarray*}

\section{Proof of Theorem \ref{theo:QR-spin}}

Let $(M,\omega,\Phi)$ be a compact Hamiltonian $K$-manifold which is $\spin$ prequantized. 
 Let $\sigma$ be the smallest face of the Weyl chamber so that 
 $\Phi(M)\cap\tgot^*_+\subset\overline{\sigma}$. Let $\mu$ be a dominant weight, and let 
$\mm_\mu$ be the multiplicity of $V_\mu^K$ in $\spinq^K(M)$.

Let $\Ocal_\mu$ be the coadjoint orbit $K\cdot(\mu+\rho)$. Since the dual representation 
$(V_\mu^K)^*$ can be realized\footnote{$\overline{\Ocal_\mu}$ is the coadjoint orbit $\Ocal_\mu$ with the opposite 
symplectic structure.} as $\spinq^K(\overline{\Ocal_\mu})$, we know by the shifting trick that 
$$
\mm_\mu= \left[\spinq^K(M\times \overline{\Ocal_\mu})\right]^K.
$$

Now  we work with the Hamiltonian $K$-manifold $N=M\times\overline{\Ocal_\mu}$ with moment 
map $\Phi_N(m,\xi)=\Phi(m)-\xi$. The Witten's deformation on $N$ gives
$\spinq^K(M\times \overline{\Ocal_\mu})=
\sum_{\beta\in\Bcal^\mu}\Qcal^K_\beta(M\times \overline{\Ocal_\mu})$ where $\Bcal^\mu$ is a finite set parametrizing 
$\Cr(\|\Phi_N\|^2)$.
We have then 
\begin{equation}\label{eq:witten-beta}
\mm_\mu= \sum_{\beta\in\Bcal^\mu}\left[\Qcal^K_\beta(M\times \overline{\Ocal_\mu})\right]^K.
\end{equation}
 We remark that $0$ does not appears in $\Bcal^\mu$ when 
$\sigma\neq \tgot^*_+$, since $\mu+\rho\notin\Phi(M)$.

The main point of this section is the following 

\medskip

\begin{prop}\label{prop:key-identity}
$\bullet$ If $\mu\notin\overline{\sigma}$, the identity (\ref{eq.RR-beta=0.bis}) holds on $C_\beta$ 
for any $\beta\in \Bcal^\mu$.  Hence $\mm_\mu=0$.

$\bullet$ If $\mu\in\overline{\sigma}$, the identity (\ref{eq.RR-beta=0.bis}) holds on $C_\beta$ for any $\beta\neq -\rho_\sigma$. 
Then 
$$
\mm_\mu=\left[\Qcal^K_{-\rho_\sigma}(M\times \overline{\Ocal_\mu})\right]^K.
$$
\end{prop}

\medskip

When $\sigma=\tgot^*_+$, we have $\rho_\sigma=0$ and Proposition \ref{prop:key-identity} tell us that 
the multiplicity $\mm_\mu$ id equal to $\left[\Qcal^K_{0}(M\times \overline{\Ocal_\mu})\right]^K$ for any $\mu\in\what{K}$. In particular 
$\mm_\mu=0$ if $\mu+\rho\notin\Phi(M)$.

When $\sigma\neq\tgot^*_+$ and $\mu\in\overline{\sigma}$, we  precise Proposition \ref{prop:key-identity} as follow.  The generalized character 
$\Qcal^K_{-\rho_\sigma}(M\times \overline{\Ocal_\mu})$  is defined as the index of a transversally elliptic symbol 
living in a neighborhood of 
$$
C_{-\rho_\sigma}=K\left(N^{\rho_\sigma}\cap\Phi_N^{-1}(-\rho_\sigma)\right).
$$

Let $K_{\rho_\sigma}$ be the stabilizer subgroup of $\rho_\sigma$. Let $W(K_{\rho_\sigma})\subset W$ be the 
Weyl subgroup of $K_{\rho_\sigma}$. A direct computation gives that
$$
C_{-\rho_\sigma}=\bigcup_{\bar{w}\in W(K_{\rho_\sigma})\backslash W}C_{-\rho_\sigma, \bar{w}}
$$
with 
$$
C_{-\rho_\sigma, \bar{w}}= K\left(M^{\rho_\sigma}\cap\Phi^{-1}(w(\mu+\rho)-\rho_\sigma)\times \{w(\mu+\rho)\}\right).
$$
We are particularly interested in the component $C_{-\rho_\sigma, \bar{e}}$. Let us denote  $C_{-\rho_\sigma, out}$ 
the union of the $C_{-\rho_\sigma, \bar{w}}$ for $\bar{w}\neq \bar{e}$. We have a decomposition 
\begin{equation}\label{eq:Z-rho-sigma}
C_{-\rho_\sigma}=C_{-\rho_\sigma, \bar{e}}\cup C_{-\rho_\sigma, out}
\end{equation}
into closed invariant disjoint subsets. Then the generalized character $\Qcal^K_{-\rho_\sigma}(M\times \overline{\Ocal_\mu})$ 
is equal to the sum
$$
\Qcal^K_{-\rho_\sigma, \bar{e}}(M\times \overline{\Ocal_\mu})+ 
\Qcal^K_{-\rho_\sigma, out}(M\times \overline{\Ocal_\mu})
$$
where both terms correspond to the specialization of the transversally elliptic symbol to the neighborhood 
of each part of the decomposition  (\ref{eq:Z-rho-sigma}).

\begin{prop}\label{prop:key-identity-Z} Suppose that $\sigma\neq\tgot^*_+$ and that $\mu\in\overline{\sigma}$.  
The identity (\ref{eq.RR-beta=0.bis}) holds on the subset $C_{-\rho_\sigma, out}$, and then 
$$
\mm_\mu=\left[\Qcal^K_{-\rho_\sigma,\bar{e}}(M\times \overline{\Ocal_\mu})\right]^K.
$$

\end{prop}

Note that $C_{-\rho_\sigma, e}=\emptyset$ if $\mu+\rho-\rho_\sigma\notin \Phi(M)$. At this stage we know then 
that $\mm_\mu=0$ if $\mu+\rho-\rho_\sigma$ does not belongs to the image of the moment map.

\subsection{Proofs of Propositions \ref{prop:key-identity} and \ref{prop:key-identity-Z}}

Let $N=M\times \overline{\Ocal_\mu}$ and let $\|\Phi_N\|^2: N\to \R$ be the square of the moment map. Recall that 
we denote by $\sigma$ the smallest face of the Weyl chamber so that $\Phi(M)\cap\tgot^*_+\subset\overline{\sigma}$. 

We want to prove that for any $n=(m,\xi)\in \Cr(\|\Phi_N\|^2)$ the vector $\beta:=\Phi(m)-\xi$ satisfies
$$
{\bf (I)}\qquad\qquad \parallel\beta\parallel^2 +\frac{1}{2}\tr_{\beta}|\T_n N | \geq\tr_{\beta}|\kgot|.
$$
Afterwards we will discuss the case of equality in ${\bf (I)}$.

The tangent space $\T_\xi\Ocal_\mu$ is equal to the $\kgot_\xi$-module $\kgot/\kgot_\xi$: then 
\begin{eqnarray*}
\tr_{\beta}|\T_\xi\Ocal_\mu|&=&\tr_{\beta}|\kgot|-\tr_{\beta}|\kgot_\xi | \\
&=&\tr_{\beta}|\kgot|,
\end{eqnarray*}
since $\beta$ belongs to the abelian subalgebra $\kgot_\xi$. 
Using that $\tr_{\beta}|\T_n N |=\tr_{\beta}|\T_m M | + \tr_{\beta}|\kgot|$, we see that ${\bf (I)}$ 
is equivalent to 
$$
{\bf (II)}\qquad\qquad \parallel\beta\parallel^2 +\frac{1}{2}\tr_{\beta}|\T_m M | \geq\frac{1}{2}\tr_{\beta}|\kgot|.
$$

The module $\kgot/\kgot_m$ is naturally a subspace of $\T_m M$. Let $E_m$ be a $K_m$-equivariant supplement 
to $\kgot/\kgot_m$ in $\T_m M$. Using that $\tr_{\beta}|\T_m M |= 
\tr_{\beta}|\kgot/\kgot_m | + \tr_{\beta}|E_m |$, we see that ${\bf (II)}$ is equivalent to  
$$
{\bf (III)}\qquad\qquad \parallel\beta\parallel^2 +\frac{1}{2}\tr_{\beta}|E_m | \geq\frac{1}{2}\tr_{\beta}|\kgot_m|.
$$

Thanks to the inclusion $\kgot_m\subset\kgot_{\Phi(m)}$, we see that ${\bf (I)}\Leftrightarrow{\bf (II)}\Leftrightarrow{\bf (III)}$ 
are induced by the following inequality
$$
{\bf (IV)}\qquad\qquad \qquad \parallel\beta\parallel^2 \geq\frac{1}{2}\tr_{\beta}|\kgot_{\Phi(m)}|.
$$

\begin{lem}\label{lem:ineq-4}
$\bullet$ For any $(m,\xi)\in \Cr(\|\Phi_N\|^2)$ the vector $\beta:=\Phi(m)-\xi$ satisfies the 
inequality {\bf (IV)}.

$\bullet$ Let $(m,\xi)\in \Cr(\|\Phi_N\|^2)$ such that  $\beta:=\Phi(m)-\xi$ satisfies the 
$ \parallel\beta\parallel^2 =\frac{1}{2}\tr_{\beta}|\kgot_{\Phi(m)}|$. Then there exists a face $\tau$ of 
$\sigma$ such that 
\begin{enumerate}
\item  $\mu\in\overline{\tau}$
\item  $(m,\xi)$ belongs to the $K$-orbit of $\Phi^{-1}(\mu+\rho-\rho_\tau)\times \{\mu+\rho\}\subset N$.
\item $\beta$ belongs to the coadjoint orbit $K\cdot(-\rho_\tau)$.
\end{enumerate}
\end{lem}

\begin{proof} Up to the multiplication of $(m,\xi)$ by an element of $K$, we can assume that $\beta\in\tgot^*$. 
Up to the multiplication of $n=(m,\xi)$ by an element of the stabilizer subgroup $K_\beta:=\{k\in K\,\vert\, 
{\rm Ad}(k)\beta=\beta\}$ 
we can assume that $n=(m,w(\mu+\rho))$ with $m\in M^\beta$ and $\Phi(m)=\beta+w(\mu+\rho)\in\tgot^*$. 

Up to the multiplication of $n=(m,w(\mu+\rho))$ by an element of the Weyl group, we can assume that 
$\Phi(m)$ belongs to the Weyl chamber: let $\tau$ be the face of $\sigma$ 
containing $\Phi(m)$ so that $K_{\Phi(m)}=K_\tau$.

So we have to prove that for $\Phi(m)=a\in\tau$ and $w\in W$ the vector $\beta=a-w(\mu+\rho)$ 
satisfies the relation
\begin{equation}\label{eq:lem-1}
\parallel\beta\parallel^2 \geq \frac{1}{2}\tr_{\beta}|\kgot_\tau|.
\end{equation}

The inequality (\ref{eq:lem-1}) is the consequence of three basic inequalities.

We have
\begin{equation}\label{eq:inegalite-1}
\|a-w(\mu+\rho)\|\geq \| a-(\mu+\rho)\| 
\end{equation}
for any $w\in W$, and (\ref{eq:inegalite-1}) is strict unless $w\in W(K_\tau)$. In order to prove (\ref{eq:inegalite-1}), 
we consider the function $\xi\in K\cdot (\mu+\rho)\longmapsto \|\xi-a\|^2=\|a\|^2+\|b\|^2-2(\xi,a)$.
It is a classical result of symplectic geometry that the function $\xi\mapsto (\xi,a)$ has 
a unique maximum on the coadjoint orbit $K\cdot (\mu+\rho)$ which is reached on an orbit of the stabilizer subgroup 
$K_a=K_\tau$. Since $a$ and $\mu+\rho$ belongs to the Weyl chamber, one checks easily that this maximum  is obtained on the 
orbit $K_\tau(\mu+\rho)$. Hence $\|k(\lambda+\rho)-a\|^2\geq\|(\lambda+\rho)-a\|^2$ for any $k\in K$ with equality only if 
$k\in K_\tau$. Inequality (\ref{eq:inegalite-1}) is proved by taking 
$k=w$.

On the other hand we have 
\begin{eqnarray}\label{eq:inegalite-2}
\|\mu+\rho-a\|&\geq& \frac{(\mu+\rho-a,\rho_\tau)}{\|\rho_\tau\|}\nonumber\\
&=& \frac{1}{\|\rho_\tau\|}\underbrace{(\mu,\rho_\tau)}_{\geq 0}+
\frac{1}{\|\rho_\tau\|}\underbrace{(\rho-\rho_\tau-a,\rho_\tau)}_{=0}+ 
\frac{1}{\|\rho_\tau\|}(\rho_\tau,\rho_\tau)\nonumber\\
&\geq & \|\rho_\tau\|.
\end{eqnarray}
Note that (\ref{eq:inegalite-2}) is strict unless $\mu\in\overline{\tau}$ and $\mu+\rho-a=\rho_\tau$. The third inequality is 
\begin{equation}\label{eq:inegalite-3}
\frac{1}{2}\tr_{\beta}|\kgot_\tau|\leq \|\rho_\tau\|\, \|\beta\|.
\end{equation}
See Example \ref{example:Tr-beta}. If we put (\ref{eq:inegalite-1}),  (\ref{eq:inegalite-2}) and  (\ref{eq:inegalite-3}) together we have
$$
\parallel\beta\parallel^2\,\geq\, \parallel\beta\parallel\, \| a-(\mu+\rho)\|\,\geq \,
\parallel\beta\parallel \|\rho_\tau\| \,\geq \, \frac{1}{2}\tr_{\beta}|\kgot_\tau|,
$$
and the equality $\parallel\beta\parallel^2=\frac{1}{2}\tr_{\beta}|\kgot_\tau|$ holds if and  
only if we have the equality in (\ref{eq:inegalite-1}),  (\ref{eq:inegalite-2}) and  (\ref{eq:inegalite-3}). 

But equalities in (\ref{eq:inegalite-1}) and (\ref{eq:inegalite-2}) gives that $w\in W(K_\tau)$, $\mu\in\overline{\tau}$ 
and $a=\mu+\rho-\rho_\tau\in\tau$. Then $(m,w(\mu+\rho))= w(m',\mu+\rho)$ with 
$\Phi(m')=w^{-1}(\mu+\rho-\rho_\tau)=\mu+\rho-\rho_\tau$ and $\beta= \mu+\rho-\rho_\tau-w(\mu+\rho)=
-w\rho_\tau$. We have then   
$$
\frac{1}{2}\tr_{\beta}|\kgot_\tau|=\frac{1}{2}\tr_{\rho_\tau}|\kgot_\tau|= \|\rho_\tau\|^2
$$
which is the equality in (\ref{eq:inegalite-3}).

\end{proof}

\medskip

Since the strict inequality in {\bf (IV)} implies the strict inequality in {\bf (I)}, Lemma \ref{lem:ineq-4} tells us that the 
identity (\ref{eq.RR-beta=0.bis}) holds on $C_\beta$ for all $\beta\in\Bcal^\mu$ when $\mu\notin\overline{\sigma}$. When 
$\mu\in\overline{\sigma}$ the identity 
(\ref{eq.RR-beta=0.bis}) holds 
\begin{enumerate}
\item  on $C_\beta$ for the $\beta$ which are not  in $K\cdot(-\rho_\tau)$, where $\tau$ is a face of $\sigma$ such that $\mu\in\overline{\tau}$, 
\item on $C_{-\rho_\sigma,w}$ for all the $\bar{w}\neq \bar{e}$.
\end{enumerate}

The proof of Proposition \ref{prop:key-identity} is completed by

\begin{lem}
Let $\tau$ be a face of $\sigma$, distinct from $\sigma$, such that $\mu\in\overline{\tau}$. Then the identity (\ref{eq.RR-beta=0.bis}) 
holds for $C_\beta$ for $\beta=-\rho_\tau$.
\end{lem}

\begin{proof} Let $\beta=-\rho_\tau$. The critical set $C_{-\rho_\tau}:=K(N^{\rho_\tau}\cap\Phi_N^{-1}(-\rho_\tau))$   admits the 
decomposition $C_{-\rho_\tau}=\cup_{w\in W}C_{-\rho_\tau, w}$ where 
$$
C_{-\rho_\tau, w}= K\left(M^{\rho_\tau}\cap\Phi^{-1}(w(\mu+\rho)-\rho_\tau)\times \{w(\mu+\tau)\}\right).
$$
It is not hard to see that $C_{-\rho_\tau, w}$ intersects $C_{-\rho_\tau, e}$ only if $w\in W(K_{\rho_\tau})$, and 
that $C_{-\rho_\tau, w}=C_{-\rho_\tau, e}$ when $w\in W(K_{\rho_\tau})$. We know then from Lemma \ref{lem:ineq-4} that the strict inequality 
in {\bf (IV)} holds on $C_{-\rho_\tau, w}$ for $w\notin W(K_{\rho_\tau})$.

Let us consider now the case where $m\in M^{\rho_\tau}\cap\Phi^{-1}(\mu+\rho-\rho_\tau)$.
We know that the equality holds in ${\bf (IV)}$ for $(m,\mu+\rho)$. 
The equality in ${\bf (I)}$ for $(m,\mu+\rho)$ is then equivalenty to 
\begin{equation}\label{eq:E-beta}
\tr_{\beta}|E_m| +\tr_{\beta}|\kgot_\tau/\kgot_m|=0.
\end{equation}

Let us prove that (\ref{eq:E-beta}) can not holds. The image of $m$ by the moment map belongs to $\tau$. 
Then $m$ belongs to the symplectic slice 
$Y_\tau\subset M$. A neighborhood $m$ is then $K\times_{K_\tau}Y_\tau$. So the tangent space at $m$ 
decomposes in two manners
\begin{eqnarray*}
\T_m M&=&\kgot/\kgot_\tau\oplus \T_m Y_\tau\\
&=& \kgot/\kgot_m \oplus E_m
\end{eqnarray*}

If (\ref{eq:E-beta}) holds we see that $\tr_{\beta}|\T_m Y_\tau|=\tr_{\beta}|E_m|=0$, which means that 
$\beta=-\rho_\tau$ acts trivially on the tangent space $\T_m Y_\tau$. Hence it would implies that 
$\rho_\tau$ acts trivially on the manifold $Y_\tau$. Since $Y_\sigma\subset Y_\tau$, the action of 
$\rho_\tau$ on the principal slice $Y_\sigma$ is also trivial.

We know that $[\kgot_\sigma,\kgot_\sigma]$ acts trivially on $Y_\sigma$: since $\rho_\sigma\in 
[\kgot_\sigma,\kgot_\sigma]$, the infinitesimal action of $\rho_\sigma$ is trivial on $Y_\sigma$. Finally if 
(\ref{eq:E-beta}) holds, we have that
$$
\rho_{\tau/\sigma}:=\rho_\tau-\rho_\sigma\in\R\sigma
$$
acts trivially on $Y_\sigma$. Note that $\rho_{\tau/\sigma}$ is a sum of weights which are orthogonal to 
$\tau$.

The moment polytope of $M$, $\Delta(M)$, which is equal to the closure of $\Phi(Y_\sigma)\subset\sigma$ 
is a convex polytope. Since the action of $\rho_{\tau/\sigma}$ is trivial on $Y_\sigma$ we knows that 
the map $\xi\in \Delta(M)\mapsto (\xi,\rho_{\tau/\sigma})$ is constant.

Finally we can use the last information in our hands: $\mu+\rho-\rho_\tau=\Phi(m)$ belongs to $\Delta(M)$. Then
for $\xi\in \Delta(M)$ we have
$$
(\xi,\rho_{\tau/\sigma})=(\mu+\rho-\rho_\tau,\rho_{\tau/\sigma})=0,
$$
since $\mu+\rho-\rho_\tau\in\tau$ and $\rho_{\tau/\sigma}\in \tau^\perp$. 
It is contradictory with the fact that $(\xi,\rho_{\tau/\sigma})= (\xi,\rho_{\sigma})>0$ for any $\xi\in\sigma$.

We have finally proved that when $(m,\xi)\in N^{\rho_\tau}\cap\Phi_N^{-1}(-\rho_\tau)$ the vector 
$\beta=\Phi(m)-\xi$ satisfies $\parallel\beta\parallel^2 +
\frac{1}{2}\tr_{\beta}|\T_m M|>\tr_{\beta}|\kgot|$.
\end{proof}

\subsection{Computation of the multiplicities when $\sigma=\tgot^*_+$}

In this section we suppose that the moment polytope $\Delta(M)=\Phi(M)\cap\tgot^*_+$ intersects the interior 
of the Weyl chamber. Let $\Delta(M)^o\subset (\tgot^*_+)^o$ be the relative interior of the moment polytope. 
We know that $\mm_\mu= [\Qcal^K_0(M\times \overline{\Ocal_\mu})]^K$ for any $\mu\in\what{K}$. 
In Definition \ref{def:indice-reduction}, we have defined the number $\Qcal(M_\mu^{\tgot^*_+})$ has follows. 
If $\mu+\rho\notin\Delta(M)^o$, we set $\Qcal(M_{\mu,\tgot^*_+})=0$. 
If $\mu+\rho\in \Delta(M)^o$, we consider, for $\esp$ generic and small enough, 
the orbifold reduced space $M_{\mu+\esp}^{\tgot^*_+}:=\Phi^{-1}(\mu+\esp+\rho)/T$ and the  
orbifold line bundle
$$
\tilde{\Lcal}_{\mu+\esp}= \left(\tilde{L}\vert_{\Phi^{-1}(\mu+\esp+\rho)}\otimes\C_{-\mu}\right)/T.
$$
The Spin quantization $\spinq(M_\mu^{\tgot^*_+})\in\Z$ is defined as the Riemann-Roch number
$$
RR(M_{\mu+\esp}^{\tgot^*_+},\tilde{\Lcal}_{\mu+\esp}).
$$

The main result of this section is the following 

\begin{theo}\label{theo:QR-1} The number 
$\left[\Qcal^K_0(M\times \overline{\Ocal_\mu})\right]^K$ 
is equal to $ \Qcal_{\hbox{{\rm \tiny spin}}}(M_\mu^{\tgot^*_+})$.  
\end{theo}

\begin{proof} When $\mu+\rho\notin\Delta(M)$, 
we see that $\Qcal^K_0(M\times \overline{\Ocal_\mu})=0$ since the moment map on 
$M\times \overline{\Ocal_\mu}$ does not goes through $0\in\kgot^*$. We have then 
$[\Qcal^K_0(M\times \overline{\Ocal_\mu})]^K=\spinq(M_\mu^{\tgot^*_+})=0$.

We consider now a dominant weight $\mu$ such that $\mu+\rho\in\Delta(M)$. Let $Y=\Phi^{-1}((\tgot^*_+)^o)$ 
be the symplectic slice with its canonical symplectic form $\omega_Y$. The action of $T$ on $(Y,\omega_Y)$ is 
Hamiltonian with moment map $\Phi_Y:=\Phi\vert_Y-\rho$. We know that 
$\tilde{L}\vert_Y$ Spin-prequantizes  $(Y,\omega_Y,\Phi_Y)$ (see Lemma \ref{lem:restriction-spin-slice}).

We consider the Riemann-Roch character $RR^T_0(Y,\tilde{L}\vert_Y\otimes\C_{-\mu})$ 
which is localised near $(\Phi_Y-\mu)^{-1}(0)\subset Y$. Thanks to the induction formula (\ref{eq:induction}), we know that
\begin{eqnarray*}
\mm_\mu=\left[\Qcal^K_0(M\times \overline{\Ocal_\mu})\right]^K&=&
\left[RR^K_0(M\times \overline{\Ocal_\mu},\tilde{L}\boxtimes\C_{[-\mu]})\right]^K\\
&=&\left[RR^T_0(Y,\tilde{L}\vert_Y\otimes\C_{-\mu})\right]^T\\
&=&\left[RR^T_{\Phi_{_Y}-\mu}(U,\tilde{L}\vert_U\otimes\C_{-\mu})\right]^T
\end{eqnarray*}
where $U$ is a small neighborhood of $\Phi_Y^{-1}(\mu)$ in $Y$.

The computation of the expression $[RR^T_{\Phi_{_Y}-\mu}(U,\tilde{L}\vert_U\otimes\C_{-\mu})]^T$ is identical 
to what we have done in Section \ref{sec:torus-case}. Forr $\esp$ small enough and generic,  we get  
\begin{eqnarray*}
[RR^T_{\Phi_{_Y}-\mu}(U,\tilde{L}\vert_U\otimes\C_{-\mu})]^T&=&[RR^T_{\Phi_{_Y}-\mu-\esp}(U,\tilde{L}\vert_U\otimes\C_{-\mu})]^T\\
&=&RR(M_{\mu+\esp}^{\tgot^*_+}, \tilde{\Lcal}_{\mu+\esp})\\
&=&\spinq(M_{\mu}^{\tgot^*_+}).
\end{eqnarray*}

When $\mu+\rho$ does not belong to the relative interior of $\p(M)$, we can choose $\esp$ so that $\mu+\rho+\esp\notin\p(M)$, and 
then $RR(M_{\mu+\esp}^{\tgot^*_+}, \tilde{\Lcal}_{\mu+\esp})=\spinq(M_{\mu}^{\tgot^*_+})=0$. 
\end{proof}

\subsection{Computation of the multiplicities when $\sigma\neq\tgot^*_+$}

Let $\mu\in\overline{\sigma}$ so that $\mu+\rho-\rho_\sigma\in\sigma$. In the rest of this section 
the term $\beta$ is $-\rho_\sigma$.

Let $\Qcal^K_{\beta, \bar{e}}(M\times \overline{\Ocal_\mu})$ be 
the generalized character defined as the index of a transversally elliptic symbol  defined in a neighborhood of 
$$
C_{\beta, \bar{e}}= K\left(M^{\beta}\cap\Phi^{-1}(\mu+\rho-\rho_\sigma)\times \{\mu+\rho\}\right)\subset N.
$$
We have proved in the last section that
$\mm_\mu=\left[\Qcal^K_{\beta,\bar{e}}(M\times \overline{\Ocal_\mu})\right]^K$.

\medskip

First we notice that the character $\Qcal^K_{\beta,\bar{e}}(M\times \overline{\Ocal_\mu})$ corresponds to 
the Riemann-Roch character $RR^K_{\beta,\bar{e}}(N,\tilde{L}_N)$ 
localized with the Kirwan vector field near $C_{\beta, \bar{e}}\subset\Cr(\|\Phi_N\|^2)$. 
We can look at $N$ as a $K_{\beta}$-Hamiltonian manifold, and consider the Riemann-Roch character 
$$
RR^{K_\beta}_{\beta,e}(N,-)
$$
localized with the Kirwan vector field 
near $C_\beta':=K_{\beta}(\Phi^{-1}(\mu+\rho-\rho_\sigma)\times \{\mu+\rho\})$.

We have prove in \cite{pep-RR} that 
\begin{equation}\label{eq:induction1}
RR^K_{\beta,e}(N,\tilde{L}_N)=\indB\left(RR^{K_\beta}_{\beta,e}(N,\tilde{L}_N)
\wedge_\C^\bullet(\kgot/\kgot_\beta)_\C\right)
\end{equation}
where $\indB: \Rfor(K_{\beta})\to \Rfor(K)$ is the induction map, and 
$(\kgot/\kgot_{\beta})_\C$ is the  complexification of the real $K_{\beta}$-module 
$\kgot/\kgot_{\beta}$.  It gives that 
$$
\left[RR^K_{\beta,e}(N,\tilde{L}_N)\right]^K=
\left[RR^{K_\beta}_{\beta,e}(N,\tilde{L}_N)\wedge_\C^\bullet(\kgot/\kgot_\beta)_\C\right]^{K_\beta}.
$$

\medskip

Let $Y_\sigma$ be the principal symplectic slice of $M$. Recall that the subgroup $[K_\sigma,K_\sigma]$ acts 
trivially on $Y_\sigma$ and that $\rho_\sigma$ belongs to $[\kgot_\sigma,\kgot_\sigma]$: hence 
$$
\Phi^{-1}(\mu+\rho-\rho_\sigma)\subset Y_\sigma\subset M^\beta
$$ 
and then $C_\beta'=\Phi^{-1}(\mu+\rho-\rho_\sigma)\times \{\mu+\rho\}$. We are looking at a $K_\beta$-invariant 
neighborhood $\Ucal$ of $C_\beta'$ in $N^\beta$. We consider the open neighborhood 
$K\times_{K_\sigma}Y_\sigma$ of $\Phi^{-1}(\mu+\rho-\rho_\sigma)$ in $M$. Since 
$K_{\beta}\cap K_\sigma= T$, one sees that 
$$
K_{\beta}\times_{T}Y_\sigma \subset \left(K\times_{K_\sigma}Y_\sigma\right)^{\beta}
$$
is a $K_{\beta}$-invariant neighborhood of $Y_\sigma$ in $M^\beta$. Then we can take 
$$
\Ucal:= \left(K_\beta\times_{T}Y_\sigma\right) \, \times\, K_\beta(\mu+\rho)
\ \subset \ N^\beta.
$$

We look at $\Ucal$ as a Hamiltonian $K_\beta$-manifold with moment map 
$\Phi_\Ucal([k,y],\xi)=k\Phi(y)-\xi\in\kgot_\beta^*$. The set 
$C'_{\beta}$ is a connected component of critical points of $\Cr(\|\Phi_\Ucal\|^2)$, and we consider the Riemann-Roch 
character 
$$
RR^{K_\beta}_{\beta}(\Ucal,-)
$$  
localized with the Kirwan vector field near $C'_{\beta}\subset\Ucal$.

Let $\Ncal$ be the normal bundle of $\Ucal$ in $N$. We have 
$\Ncal=\Ncal_1\boxtimes \Ncal_2$ where $\Ncal_1$ is the normal bundle of $K_{\beta}\times_{T}Y_\sigma$ 
in $K\times_{K_{\sigma}}Y_\sigma$ and $\Ncal_2$ is the normal bundle of $K_{\beta}(\mu+\rho)$ 
in $K(\mu+\rho)$.  One computes that $\Ncal_1= K_{\beta}\times_{T} N_1$ where 
$$
N_1= \sum_{\stackrel{\alpha>0}{\alpha\vert_\sigma\neq 0,\,(\alpha,\beta)\neq 0}} \kgot(\alpha),
$$
and that $\Ncal_2= K_{\beta}\times_{T} N_2$ where 
$$
N_2= \sum_{\stackrel{\alpha<0}{(\alpha,\beta)\neq 0}} \kgot(\alpha).
$$

We decompose $\Ncal$ in the sum of the polarized bundle $\Ncal^{+,\beta}$ and $\Ncal^{-,\beta}$. 
Similarly let $\Ncal_\C$ the complexified bundle, and its polarized $\beta$-positive part 
$\Ncal_\C^{+,\beta}$.

Let $S(\Ncal_\C^{+,\beta})=\sum_{k\geq 0}S^k(\Ncal_\C^{+,\beta})$ be the symmetric algebra 
vector bundle associated to $\Ncal_\C^{+,\beta}$. Let us compute the rank $n_{\beta,+}$ of the 
polarized vector bundle vector $\Ncal^{+,\beta}$. We have 
\begin{eqnarray*}
n_{\beta,+}&=&\sharp\left\{\alpha>0\  \vert \ (\alpha,\beta)> 0\ {\rm and} \ \alpha\vert_\sigma\neq 0\right\}+
\sharp\left\{\alpha<0\  \vert \  (\alpha,\beta)> 0\right\}\\
&=&\sharp\left\{\alpha>0\  \vert \ (\alpha,\beta)> 0\ \right\}+
\sharp\left\{\alpha<0\  \vert \  (\alpha,\beta)> 0\right\} \qquad\qquad \qquad(1)\\
&=& \frac{1}{2}\dim(K/K_\beta).
\end{eqnarray*}

In $(1)$ we use that $\alpha\vert_\sigma=0$ imposes $(\alpha,\rho-\rho_\sigma)=0$. Then 
$(\alpha,\beta)=-(\alpha,\rho)<0$ for $\alpha>0$. Let $\det \Ncal^{+,\beta}$ be the 
determinant line bundle associated to $\Ncal^{+,\beta}$. 

Thanks to the results in 
\cite{pep-RR}[Section 6.3], we know that 
\begin{equation}\label{eq:localisation-beta}
RR^{K_\beta}_{\beta, e}(N,\tilde{L}_N)=(-1)^{n_{\beta}^+}
RR^{K_\beta}_{\beta}\left(\Ucal,\tilde{L}_N\vert_{\Ucal}
\otimes \det \Ncal^{+,\beta}\otimes S(\Ncal_\C^{+,\beta})\right).
\end{equation}

Hence we know that $\mm_\mu=\left[RR^K_{\beta,\bar{e}}(N,\tilde{L}_N)\right]^K$ is equal to 
$(-1)^{n_{\beta}^+}$ times
\begin{eqnarray}\label{eq:induction2}
&&\left[RR^{K_\beta}_{\beta}\left(\Ucal,\tilde{L}_N\vert_{\Ucal}
\otimes \det \Ncal^{+,\beta}\otimes S(\Ncal_\C^{+,\beta})\right)
\wedge_\C^\bullet(\kgot/\kgot_{\beta})_\C\right]^{K_\beta}\nonumber
\\
&=& \sum_{k\geq 0}\left[RR^{K_\beta}_{\beta}\left(\Ucal,\tilde{L}_N\vert_{\Ucal}
\otimes \det \Ncal^{+,\beta}\otimes S^k(\Ncal_\C^{+,\beta})\right)
\wedge_\C^\bullet(\kgot/\kgot_{\beta})_\C\right]^{K_\beta}
\end{eqnarray}

Let $E\to \Ucal$ be any $K_\beta$-equivariant Hermitian vector bundle. Since $\beta$ acts 
trivially on $\Ucal$ we can look at the Lie derivative $\Lcal(\beta)$ on $E$. Then 
$\frac{1}{i}\Lcal(\beta)$ defines for each $x\in \Ucal$ a Hermitian endomorphism of $E_x$. 
Let us denote 
$$
\frac{1}{i}\Lcal(\beta) >0
$$
when all its eigenvalue on the fibers of $E$ are stricly positive. 

We made in \cite{pep-RR} the crucial observation

\begin{lem}\label{lem:indice-beta-positif}
If $\frac{1}{i}\Lcal(\beta) >0$, then 
$\left[RR^{K_\beta}_{\beta}(\Ucal,E)\right]^{K_\beta}=0$.
\end{lem}

Let us compute the Lie action $\Lcal(\beta)$ on the fibers of the bundle 
$\tilde{L}_N\vert_{\Ucal}\otimes \det \Ncal^{+,\beta}\otimes S^k(\Ncal_\C^{+,\beta})$. 
It is easy to check (see \cite{pep-ENS}) that on $\tilde{L}_N\vert_{\Ucal}\otimes \det \Ncal^{+,\beta}$  
the Lie action $\frac{1}{i}\Lcal(\beta)$ is equal to 
$$
\|\beta\|^2+ \frac{1}{2}\tr_{\beta}| \Ncal | 
$$

Look now at the Lie derivative $\Lcal(\beta)$ on $\wedge_\C^\bullet(\kgot/\kgot_\beta)_\C$. 
As a $T$-module $\wedge_\C^\bullet(\kgot/\kgot_\beta)_\C$ is equal to 
\begin{eqnarray*}
\prod_{(\alpha,\beta)\neq 0}(1-e^{i\alpha})&=&\prod_{(\alpha,\beta)<0}(1-e^{i\alpha})
\prod_{(\alpha,\beta)>0}(1-e^{i\alpha})\\
&=& (-1)^{1/2\dim (K/K_\beta)}e^{-i\delta_\beta}
\Big(\prod_{(\alpha,\beta)>0}(1-e^{i\alpha})\Big)^2
\end{eqnarray*}
with $\delta_\beta=\sum_{(\alpha,\beta)>0}\alpha$. Notice that 
$e^{-i\delta_\beta}$ defines a character of the group $K_\beta$ that we denote 
$\C_{-\delta_\beta}$. We have proved then that 
$$
\wedge_\C^\bullet(\kgot/\kgot_\beta)_\C= (-1)^{n_{\beta,+}}\C_{-\delta_\beta} \oplus R
$$
where the Lie derivative $\frac{1}{i}\Lcal(\beta)$ on $\C_{-\delta_\beta}$ is equal to 
$-(\delta_\beta,\beta)=-\tr_{\beta}| \kgot |$ and the Lie derivative $\frac{1}{i}\Lcal(\beta)$ on 
the $\kgot_\beta$-module $R$ is $>-\tr_{\beta}| \kgot |$.

\medskip 

Since $\|\beta\|^2+ \frac{1}{2}\tr_{\beta}| \Ncal | =\tr_{\beta}| \kgot |$, we can conclude  that 
the Lie derivative $\frac{1}{i}\Lcal(\beta)$

\begin{enumerate}
\item is equal to zero on $\tilde{L}_N\vert_\Ucal\otimes \det \Ncal^{+,\beta}\otimes \C_{-\delta_\beta}$, 

\item is $>0$ on $\tilde{L}_N\vert_\Ucal\otimes \det \Ncal^{+,\beta}\otimes R$, 

\item is $>0$ on $\tilde{L}_N\vert_\Ucal\otimes \det \Ncal^{+,\beta}\otimes S^k(\Ncal_\C^{+,\beta})
\otimes \wedge_\C^\bullet(\kgot/\kgot_\beta)_\C$ for any $k\geq 1$.
\end{enumerate}

With Lemma \ref{lem:indice-beta-positif}, we see that the sum (\ref{eq:induction2}) restricts  to 
$$
(-1)^{n_{\beta,+}}\left[RR^{K_\beta}_\beta\left(\Ucal,\tilde{L}_N\vert_\Ucal
\otimes \det \Ncal^{+,\beta}\right)\otimes\C_{-\delta_\beta}
\right]^{K_\beta}
$$

\medskip

At this stage we have proved that the multiplicity $\mm_\mu$ is equal to 
\begin{equation}\label{eq:multiplicite-beta}
\left[ RR^{K_\beta}_\beta\left(\Ucal,\tilde{L}_N\vert_\Ucal \otimes \det \Ncal^{+,\beta}\right)
\otimes\C_{-\delta_\beta} \right]^{K_\beta}.
\end{equation}

On the symplectic slice $(Y_\sigma,\omega_\sigma)$, we have the moment map $\Phi_\sigma-\mu$ relative to the action of 
$Z_\sigma$. The data $(Y_\sigma,\omega_\sigma,\Phi_\sigma-\mu)$ is Spin-prequantized by the line bundle 
$\tilde{L}\vert_{Y_\sigma}\otimes \C_{-\mu}$. Let 
\begin{equation}\label{eq:RR-Z-sigma}
RR^{Z_\sigma}_0(Y_\sigma,\tilde{L}\vert_{Y_\sigma}\otimes \C_{-\mu})\in \Rfor(Z_\sigma) 
\end{equation}
be the Riemann-Roch character localized near $(\Phi_\sigma-\mu)^{-1}(0)=\Phi^{-1}(\mu+\rho-\rho_\sigma)\subset Y_\sigma$.

We conluce the computation of the multiplicity $\mm_\mu$ with the 

\begin{lem}
We have 
\begin{eqnarray*}
\mm_\mu&=&\left[RR^{K_\beta}_\beta\left(\Ucal,\tilde{L}_N\vert_\Ucal
\otimes \det \Ncal^{+,\beta}\right)\otimes\C_{-\delta_\beta}\right]^{K_{\beta}}\\
&=& \left[ RR^{Z_\sigma}_0(Y_\sigma,\tilde{L}\vert_{Y_\sigma}\otimes \C_{-\mu})\right]^{Z_\sigma}\quad (1)\\
&=& \Qcal_{\hbox{\rm \tiny spin}}(M_{\mu}^\sigma). \quad\quad\quad \quad\quad\quad\quad\quad(2)
\end{eqnarray*}
\end{lem}

\begin{proof} Let us prove that $(1)$ is a consequence of the induction formula of Proposition \ref{prop:induction}.  
First we notice that the data $(Y_\sigma,\omega_\sigma,\Phi_\sigma-\mu, \tilde{L}\vert_{Y_\sigma}\otimes \C_{-\mu})$ is naturally 
equipped with an action of the maximal torus, but with a trivial action of $T/Z_\sigma$. So the generalized 
character (\ref{eq:RR-Z-sigma}) coincides with 
$$
RR^{T}_0(Y_\sigma,\tilde{L}\vert_{Y_\sigma}\otimes \C_{-\mu})\in\Rfor(T).
$$

Let us consider the Hamiltonian  $K_\beta$-manifold
$\Ucal:= \left(K_\beta\times_{T}Y_\sigma\right) \, \times\, \overline{K_\beta(\mu+\rho)}$. 
Since $K_\beta$ acts trivially on $\rho_\sigma$ the map $\xi\mapsto \xi+\rho_\sigma$ realizes 
a $K_\beta$-equivariant symplectomorphic between the coadjoint orbits  $\overline{K_\beta(\mu+\rho)}$  and
$$
\overline{\Ocal}:=\overline{K_\beta(\mu+\rho-\rho_\sigma)}.
$$
The manifold $\Ucal$ is then symplectomorphic to $\left(K_\beta\times_{T}Y_\sigma\right) \times \overline{\Ocal}$. 
Moreover, one sees that the generalized Riemann-Roch character $RR^{K_\beta}_{\beta}(\Ucal,-)$ 
coincides with the Riemann-Roch character 
$$
RR^{K_\beta}_{0}(\left(K_\beta\times_{T}Y_\sigma\right) \times \overline{\Ocal},-)
$$ 
localized  on $C_0:=K_\beta(\Phi^{-1}(\mu+\rho-\rho_\sigma)\times\{\mu+\rho-\rho_\sigma\})$.

Since $K_\beta\cap K_\sigma= T$, the Hamiltonian $T$-manifold $Y_\sigma$ corresponds to  the symplectic slice of 
the Hamiltonian $K_\beta$-manifold $K_\beta\times_{T}Y_\sigma$. 

The bundle $\det \Ncal^{+,\beta}$ over $\left(K_\beta\times_{T}Y_\sigma\right) \times \overline{\Ocal}$ 
is equal to the product of $K_\beta\times_{T}\C_{\delta_1}\to K_\beta\times_{T}Y_\sigma$ 
with $K_\beta\times_{T}\C_{\delta_2}\to \overline{\Ocal}$, where
$$
\delta_1=\sum_{\stackrel{\alpha>0}{(\alpha,\beta)> 0}} \alpha\quad {\rm and}\quad 
\delta_2= \sum_{\stackrel{\alpha<0}{(\alpha,\beta) >0}} \alpha.
$$
The line bundle $\tilde{L}_N$ is equal to the product of $\tilde{L}$ with $K\times_T\C_{-\mu}$.  
Then the restrictions of the line bundle $\det \Ncal^{+,\beta}$  and $\tilde{L}_N$ to 
$Y_\sigma\times\{\mu+\rho-\rho_\sigma\}$ are respectively equal to, the trivial line bundle 
$\C_{\delta_1+\delta_2}=\C_{\delta_\beta}$, and to the line bundle $\tilde{L}\vert_{Y_\sigma}\otimes \C_{-\mu}$.

Finally the  induction formula of Proposition \ref{prop:induction} gives that 
\begin{eqnarray*}
RR^{K_\beta}_\beta\left(\Ucal,\tilde{L}_N\vert_\Ucal\otimes \det \Ncal^{+,\beta}\right)
&=&RR^{K_\beta}_{0}(\left(K_\beta\times_{T}Y_\sigma\right) \times \overline{\Ocal},
\tilde{L}_N\vert_\Ucal\otimes \det \Ncal^{+,\beta})\\
&=& {\rm Ind}^{^{K_\beta}}_{_T}\left(RR^{T}_0(Y_\sigma,\tilde{L}\vert_{Y_\sigma}\otimes \C_{-\mu})\otimes \C_{\delta_\beta}\right)\\
&=& {\rm Ind}^{^{K_\beta}}_{_T}\left(RR^{T}_0(Y_\sigma,\tilde{L}\vert_{Y_\sigma})\right)\otimes \C_{\delta_\beta}.
\end{eqnarray*}
Hence
\begin{eqnarray*}
\left[RR^{K_\beta}_\beta\left(\Ucal,\tilde{L}_N\vert_\Ucal \otimes \det \Ncal^{+,\beta}\right)
\otimes\C_{-\delta_\beta}\right]^{K_{\beta}}
&=& \left[ RR^{T}_0(Y_\sigma,\tilde{L}\vert_{Y_\sigma}\otimes \C_{-\mu})\right]^{T}\\
&=& \left[ RR^{Z_\sigma}_0(Y_\sigma,\tilde{L}\vert_{Y_\sigma}\otimes \C_{-\mu})\right]^{Z_\sigma}.
\end{eqnarray*}

Equality $(2)$, i.e.
$$
\left[ RR^{Z_\sigma}_0(Y_\sigma,\tilde{L}\vert_{Y_\sigma}\otimes \C_{-\mu})\right]^{Z_\sigma}
=\Qcal_{\hbox{\rm \tiny spin}}(M_{\mu}^\sigma), 
$$
has been proved in Section \ref{sec:torus-case}.

\end{proof}


\end{document}